\documentclass{amsart}
\usepackage[utf8]{inputenc}
\usepackage{amsmath}
\usepackage{textcase}
\usepackage{amsfonts}
\usepackage{amssymb}
\usepackage{amsthm}
\usepackage{bbm}
\usepackage{graphicx}

\renewcommand{\sqrt}[1]{\left( #1 \right)^\frac12}
\newcommand{\sqrtp}[1]{\left(#1\right)^\frac{1}{p}}

\newcommand{\V}{U}

\newcommand{\D}[0]{\mathbb{D}}
\newcommand{\R}[0]{\mathbb{R}}
\newcommand{\E}[0]{\mathbb{E}}
\newcommand{\Z}[0]{\mathbb{Z}}
\newcommand{\T}[0]{\mathbb{T}}
\newcommand{\N}[0]{\mathbb{N}}

\renewcommand{\d}[0]{\mathrm{d}}
\newcommand{\sgn}[0]{\mathrm{sgn}}

\renewcommand{\coprod}{\bigsqcup}
\newcommand{\ol}{\overline}
\newcommand{\n}[1]{\left\|#1\right\|}
\newcommand{\p}[1]{\left(#1\right)}

\newcommand{\acts}[2]{#1:#2\, \rotatebox[origin=c]{90}{$\circlearrowleft$}}
\newcommand{\id}{\mathrm{id}}
\newcommand{\diag}{\mathrm{diag}\,}

\newcommand{\beq}{\begin{equation}}
\newcommand{\eeq}{\end{equation}}
\newcommand{\supp}[0]{\mathrm{supp}\,}
\newcommand{\clspan}[0]{\overline{\mathrm{span}}\,}

\newtheorem{thm}{Theorem}[section]
\newtheorem{lem}[thm]{Lemma}

\newtheorem{deff}[thm]{Definition}
\newtheorem{cor}[thm]{Corollary}

\numberwithin{equation}{section}

\theoremstyle{plain}
\newcounter{thmintr}

\newtheorem{intrthm}[thmintr]{Theorem}

\allowdisplaybreaks

\author{Maciej Rzeszut}
\address{Institute of Mathematics of Polish Academy of Sciences\\
Śniadeckich 8\\
00-656 Warszawa}
\email{mrzeszut@impan.pl}
\title[Johnson--Schechtman inequalities and interpolation]{Weighted and multivariate Johnson--Schechtman inequalities with application to interpolation theory}

\begin{document}

\begin{abstract} We prove a weighted version of a classical inequality of Johnson and Schechtman from which we derive a decomposition theorem for $p$-th moments ($0<p\leq 1$) of nonnegative generalized $U$-statistics with constant not dependent on $p$. In particular, for $1\leq p\leq 2$, the norm in the subspace $\V^p_{\leq m}\left(\Omega^\infty\right)$ of $L^p\left(\Omega^\infty\right)$ spanned by functions dependent on at most $m$ variables is equivalent to the norm in a suitable interpolation sum of $L^p\left(L^2\right)$ spaces. As a consequence, we obtain some interpolation properties of $U^1_m\left(\Omega^\infty,\ell^p\right)$ that are known to imply cotype 2 of $L^1/\V_{\leq m}^1\left(\Omega^\infty\right)$.  
\end{abstract}

\maketitle
\tableofcontents

\section{Introduction} 

The well known inequality due to Rosenthal \cite{rosineq} states that for $1\leq p<\infty$,
\beq \label{rosintr}\left\|\sum_i X_i\right\|_{L^p(\Omega)}\simeq_p \max\left(\sum_i \left\|X_i\right\|_{L^1}, \left(\sum_i \left\|X_i\right\|_{L^p}^p\right)^\frac{1}{p}\right)\eeq
where $X_i$ are nonnegative independent random variables on $\p{\Omega,\mathcal{F},\mu}$. Originally, it was proved for the purpose of Banach space geometry. The precise growth of the constant as a function of $p$ was found in \cite{JSchZconst}. In the case of $0<p<1$, it appears that there is no known expression for $\left\|\sum X_i\right\|_{L^p}$ that would be as explicit as the right hand side of \eqref{rosintr}. Theorems providing two sided bounds for this quantity, valid for all $0<p<\infty$, were proved by Johnson and Schechtman \cite{JSch}, Klass and Nowicki \cite{klnow} and Latała \cite{ralatalamoments}; see also \cite{dilineq}. All of them contain an Orlicz norm in some form. The most important for us  is a special case of the main theorem from \cite{JSch}, namely the inequality 
\beq \label{JSchintr}\left\|\sum_i X_i\right\|_{L^p}\simeq_p \inf_{X_i=Y_i+Z_i} \sum_i \left\|Y_i\right\|_{L^1}+ \left(\sum_i \left\|Z_i\right\|_{L^p}^p\right)^\frac{1}{p},\eeq
valid for $0< p\leq 1$. It is a natural counterpart to \eqref{rosintr} in the following sense. Suppose that $\p{\mathcal{F}_i}_{i=1}^\infty$ are independent and $X_i$ is $\mathcal{F}_{i}$-measurable. Since the sequence $\p{X_i}_{i=1}^\infty$ carries the same information as a function $\bigsqcup_i X_i$ on the disjoint union $\overline{\Omega}=\bigsqcup_i \left(\Omega,\mathcal{F}_i,\mu\right)$ (which is now a sigma-finite measure space), the last two inequalities can be conveniently written as 
\beq\label{introrlicz} \left\|\sum_i X_i\right\|_{L^p}\simeq_p 
\begin{cases} \left\|\bigsqcup_i X_i\right\|_{L^1\cap L^p\left(\overline{\Omega}\right)}&\text{ for }1\leq p<\infty \\ \left\|\bigsqcup_i X_i\right\|_{L^1+L^p\left(\overline{\Omega}\right)}&\text{ for }0<p\leq 1 \end{cases} \simeq_p \left\|\bigsqcup_i X_i \right\|_{L^{\phi_p}\p{\overline{\Omega}}}\eeq
where $+$ denotes the interpolation sum and $\phi_p$ is an Orlicz function such that $\phi_p(t)\simeq t$ for $0\leq t\leq 1$ and $\phi_p(t)\simeq t^p$ for $t\geq 1$. For more information about Orlicz norms in this context, we refer the reader to \cite{dilhandbook}. 
\par
It is a common practice to search for analogues of classical theorems concerning independent random variables in the setting of $U$-statistics, introduced by Hoeffding in \cite{wassily}. This has been done for CLT (see e.g. \cite{cltustat}, \cite{jansonnowicki}), LIL (see e.g. \cite{GKLZ}, \cite{adlatlilustat}, \cite{lilustathilb}), SLLN (see e.g. \cite{hoeffslln}) just to name a few. A natural multivariate counterpart to $\left\|\sum X_i\right\|_{L^p}$ for nonnegative and independent $X_i$ is the $p$-th moment of a nonnegative generalized decoupled $U$-statistic, i.e. the quantity
\beq \label{ustatintrdec}\left\|\sum_{i_1<\ldots<i_m}f_{i_1,\ldots,i_m}\p{X^{(1)}_{i_1},\ldots,X^{(m)}_{i_m}}\right\|_{L^p}\eeq
where $X^{(j)}_{i_j}$ are independent random variables and $f_{i_1,\ldots,i_m}$ are nonnegative functions on $\mathbb{R}^m$. By virtue of a decoupling inequality due to Zinn \cite{zinn}, if the distribution of $X^{(j)}_{i}$ is the same for all $i,j$, then \eqref{ustatintrdec} is equivalent to its undecoupled version 
\beq \label{ustatintrundec}\left\|\sum_{i_1<\ldots<i_m}f_{i_1,\ldots,i_m}\p{X_{i_1},\ldots,X_{i_m}}\right\|_{L^p}.\eeq
For $p\geq 1$, two-sided bounds for \eqref{ustatintrdec} in terms of mixed $L^1\p{L^p}$ norms were developed in \cite{GLZ}. They were generalized to Banach space valued $U$-statistics in \cite{adamczak}, extending the inequalities of Rosenthal and of Klass and Nowicki. However, the authors indicated the lack of a satisfactory counterpart to these results for $0<p< 1$. For more information about $U$-statistics and decoupling we refer to \cite{GdlP}. \par
Let us shift our attention to the mean zero setting. Assuming that $\E X_i=0$, for $p\geq 1$, by Marcinkiewicz and Zygmund inequality \cite{marzyg}, we have 
\beq \label{intrmzyg}\left\|\sum X_i\right\|_{L^p}\simeq_p \left\|\sum \left|X_i\right|^2\right\|_{L^{\frac{p}{2}}}^{\frac12}\eeq
which allows to directly translate \eqref{introrlicz} to the mean zero case. A usual multivariate counterpart of independent mean zero variables are generalized canonical $U$-statistics, i.e. sums of the form
\beq \label{intrgencanundec}\sum_{i_1<\ldots<i_m }f_{i_1,\ldots,i_m}\p{X_{i_1},\ldots,X_{i_m}}\eeq
where $X_i$ are independent and identically distributed, while $f_{i_1,\ldots,i_m}$ are mean zero in each variable with respect to the law of $X_i$. By an inequality due to Bourgain \cite[Proposition 7]{bourgwalsh}, we get an analogous equivalence of $p$-th moment of \eqref{intrgencanundec} to the $p$-th moment of a square function
\beq \p{\sum_{i_1<\ldots<i_m }\left|f_{i_1,\ldots,i_m}\p{X_{i_1},\ldots,X_{i_m}}\right|^2}^\frac{1}{2}.\eeq 
Bourgain and Kwapień in \cite{bourgwalsh} and \cite{kwap} considered subspaces $\V^p_m\p{\Omega}$ of $L^p\p{\Omega}$ spanned by random variables of the form
\beq \label{intrsmallx}\sum_{i_1<\ldots<i_m }f_{i_1,\ldots,i_m}\p{X_{i_1},\ldots,X_{i_m}}\eeq
for all $f_{i_1,\ldots,i_m}\in L^1\p{\R^m}$ mean zero in each argument with respect to the law of $X$. The subspaces $\V^2_m\p{\Omega}$ for $m=0,1,2,\ldots$ form an orthogonal decomposition of $L^2\p{\Omega}$ and it turned out that $\V^p_m\p{\Omega}$ is complemented in $L^p\p{\Omega}$ for $1<p<\infty$, but not for $p\in\{1,\infty\}$. Moreover, by \eqref{intrmzyg} and \eqref{introrlicz}, $\V^p_1\p{\Omega}$ is isomorphic to $L^2\cap L^p\p{\overline{\Omega}}$ or $L^2+L^p\p{\overline{\Omega}}$ when $2\leq p<\infty$ or $1\leq p\leq 2$, respectively. If $\Omega= [0,1]$, then $U^p_1\p{\Omega}$ is isomorphic to $L^p\p{\R}$ for $1<p<\infty$, but not for $p=1$, see \cite{JMSchT}. This makes the case $p=1$ the most interesting to study. \par
Let us briefly introduce some aspects of inteprolation theory that will be of some importance to us. Let $X_1\subset L^1$ and $X_2=\p{X_1\cap L^2,\|\cdot\|_{L^2}}$. A desirable property of such a pair is $K$-closedness in $\p{L^1,L^2}$, from which one can derive real interpolation spaces between $X_1$ and $X_2$, see Section 2 for details. It is trivially satisfied if the orthogonal projection on $X_2$ is bounded in $L^1$. Bourgain proved in \cite{bourginterp} that if this projection is a Calder\'on--Zygmund operator, then  
\beq \label{intrbkcl}\p{X_1,X_2}\text{ is }K\text{-closed in }\p{L^1,L^2}.\eeq
A little is known about possible weaker assumptions on the projection onto $X_2\subset L^2$ that would imply \eqref{intrbkcl}. It has been proved in \cite{xul1h1} for an $m$-fold tensor of Riesz projection. It has also been shown in \cite{kisbian} for a tensor of a Riesz projection and a Calder\'on--Zygmund projection for the usually more difficult $(2,\infty)$ side of the interpolation scale. \par
Another interesting inteprolation property of subspaces of $L^1$ is connected to work of Bourgain \cite{bourgcotype}, Pisier \cite{pisiersimple} and Xu \cite{xul1h1}, resulting in a theorem that if $X\subset L^1$ is such that
\beq \label{intrbpx}\p{X\p{\ell^1},X\p{\ell^2}}\text{ is }K\text{-closed in }\p{L^1\p{\ell^1},L^2\p{\ell^2}},\eeq
then $L^1/X$ is of cotype 2 and every operator $L^1/X\to\ell^2$ is 1-summing. In fact, it was originally motivated by the question of cotype of $L^1/H^1$ answered by Bourgain \cite{bourgcotype}. Later, it was extended to $X=H^1\p{\mathbb{D}^m}$ in \cite{xul1h1} and for $X=H^1\p{\mathbb{B}_m}$ in \cite{bourginterp}.  \par
Let us turn to a detailed description of the main results of the paper. We are going to provide a weighted version of \eqref{JSchintr}, which in particular shows that in this inequality the constant is independent of $p$. Let us state a simplified version of Theorem \ref{premainlemma}. 
\begin{intrthm} \label{intrweighted}If $1\leq r<\infty$, and $X_i$ are independent and $W_i$ are $[0,1]$-valued weigths satisfying $\E\p{W_i\mid X_i}\geq \kappa$ for some constant $\kappa>0$ and all $i$, then
\beq \E\p{\sum_i \left|W_i X_i\right|^r}^\frac{1}{r} \geq \frac{1}{2}\kappa^r \inf_{X_i=Y_i+Z_i} \p{\sum \E Y_i^r}^\frac{1}{r}+ \sum \E Z_i.\eeq \end{intrthm}
From this, we derive the following theorem (see Theorem \ref{4summand}).

\begin{intrthm}\label{intr4sum}Let $f_{i,j}\in L^1\left(\Omega^2\right)$ for $i,j=1,\ldots,n$ and $1\leq p<\infty$. Then
\begin{eqnarray}\label{eq:4sumequiv}\int_{\Omega^n}\int_{\Omega^n}\left(\sum_{i,j}\left|f_{i,j}\left(x_i,y_j\right)\right|^p\right)^\frac{1}{p} \d x\d y\simeq \inf_{\coprod_{i,j}f_{i,j}=a+b+c+d} &\|a\|_{L^1\left(\coprod_i \Omega \times \coprod_j \Omega\right)}\\ \nonumber +&\|b\|_{L^p\left(\coprod_i \Omega \times \coprod_j \Omega\right)}\\ \nonumber + &\|c\|_{L^1\left(\coprod_i \Omega,L^p\left(\coprod_j \Omega\right)\right)}\\ \nonumber+ &\|d\|_{L^1\left(\coprod_j \Omega,L^p\left(\coprod_i \Omega\right)\right)}\end{eqnarray}
with a constant not dependent on $p$. In more explicit terms, the inequality `$\gtrsim$' means that if \beq \label{eq:lesssim1}\int_{\Omega^n}\int_{\Omega^n}\left(\sum_{i,j}\left|f_{i,j}\left(x_i,y_j\right)\right|^p\right)^\frac{1}{p} \d x\d y\leq 1,\eeq then there is a decomposition \beq \label{eq:decompabcd}f_{i,j}=a_{i,j}+b_{i,j}+c_{i,j}+d_{i,j}\eeq such that 
\beq \sum_{i,j}\int_{\Omega}\int_{\Omega}\left|a_{i,j}\left(\xi ,\upsilon \right)\right| \d \xi \d \upsilon\lesssim 1,\eeq
\beq \p{\sum_{i,j}\int_{\Omega}\int_{\Omega}\left|b_{i,j}\left(\xi ,\upsilon \right)\right|^p\d \xi \d \upsilon }^\frac{1}{p}\lesssim 1,\eeq
\beq \sum_i \int_{\Omega}\left(\sum_j \int_{\Omega}\left|c_{i,j}\left(\xi ,\upsilon \right)\right|^p\d \upsilon \right)^\frac{1}{p} \d \xi \lesssim 1,\eeq
\beq \sum_j \int_{\Omega}\left(\sum_i \int_{\Omega}\left|d_{i,j}\left(\xi ,\upsilon \right)\right|^p\d \xi \right)^\frac{1}{p} \d \upsilon \lesssim 1.\eeq 
Moreover, it can be chosen in such a way that for each $i,j$, supports of $a_{i,j},b_{i,j},c_{i,j},d_{i,j}$ are disjoint.\end{intrthm}
This is easily seen to be equivalent to the following, see Corollary \ref{4summandmz}. 
\begin{intrthm}\label{intr4summz}For $1\leq p\leq 2$ and $f\in U^p_2\p{\Omega^n}$ such that $f(x)=\sum_{i<j} f_{i,j}\p{x_i,x_j}$, 
\begin{align} \left\|f\right\|_{L^p\p{\Omega^n}}\simeq \inf_{f_{i,j}=a_{i,j}+b_{i,j}+c_{i,j}+d_{i,j}}& \p{\sum_{i<j}\int_{\Omega^2}\left|a_{i,j}\p{\xi,\upsilon}\right|^2\d\xi\d \upsilon}^\frac{1}{2}\\ 
+& \p{\sum_{i<j}\int_{\Omega^2}\left|b_{i,j}\p{\xi,\upsilon}\right|^p\d\xi\d \upsilon}^\frac{1}{p}\\
+& \p{ \sum_i\int_{\Omega}\p{\sum_{j>i}\int_{\Omega^2}\left|c_{i,j}\p{\xi,\upsilon}\right|^2\d \upsilon}^\frac{p}{2}\d\xi}^\frac{1}{p}\\
+& \p{ \sum_j\int_{\Omega}\p{\sum_{i<j}\int_{\Omega^2}\left|d_{i,j}\p{\xi,\upsilon}\right|^2\d\xi}^\frac{p}{2}\d \upsilon}^\frac{1}{p}.\end{align}
Moreover, the decomposition can be chosen such that $a_{i,j},b_{i,j},c_{i,j},d_{i,j}$ are mean zero in each variable. 
\end{intrthm}
Both Theorems \ref{intr4sum} and \ref{intr4summz} have natural $m$-variable extensions (see Theorems \ref{2msummand} and \ref{2msummandmz} for details). Finally, Theorems \ref{intrweighted} and \ref{intr4summz} are later utilized by us to prove that Hoeffding subspaces enjoy the mentioned interpolation properties \eqref{intrbkcl} and \eqref{intrbpx}, compare Theorems \ref{kclvml1} and \ref{bourgu1m}. 
\begin{intrthm}The couple $\p{U^1_m\p{\ell^1},U^1_m\p{\ell^2}}$ is $K$-closed in 
$\p{L^1\p{\ell^1},L^1\p{\ell^2}}$.\end{intrthm}
\begin{intrthm}The couple $\p{U^1_m,U^2_m}$ is $K$-closed in $\p{L^1,L^2}$.\end{intrthm}

The paper is organized as follows. In Section \ref{secoverview}, we recall some definitions and tools to be used later. In Section \ref{secweighted}, we prove a weighted generalization of the classical Johnson--Schechtman inequality, as a byproduct getting a new proof of the historic result with an absolute constant. This weighted inequality will be applied in Section \ref{secdecomp} to remove the obstacles that arise while iterating one-variable results leading to a Johnson--Schechtman type decomposition for low moments of nonnegative $U$-statistics. In Section \ref{secinterp}, we apply the resulting decomposition theorems to obtain results about real interpolation of $\V^p_m\p{\Omega^\infty}$ and $\V^1_m\p{\Omega^\infty,\ell^p}$ spaces between $p=1$ and $p=2$. 

\section*{Acknowledgements}Most of results of this paper are taken from my doctoral thesis \cite{diss}. I am grateful to my advisors: Fedor Nazarov and Michał Wojciechowski for their mentorship and support, especially their unending willingness to discuss my research.

\section{Overview of basic notions and facts}\label{secoverview}
\textbf{Notation.} $\gtrsim,\lesssim,\simeq$ denote $\geq,\leq,=$ respectively up to a consatnt. The expression $\|f\|_{L^p(X)}$ is defined as $\p{\int \|f\|_X^p}^\frac1p$ for all $0<p<\infty$, which makes it a norm for $p\geq 1$ and a quasinorm for $p\in (0,1)$. \par
\textbf{Probability spaces and conditional expectations.} In all of the text, $\left(\Omega,\mathcal{F},\mu\right)$ will be a probability space. We will equip sets of the form $\Omega^I$, where $I$ is an at most countable index set, with the product measure $\mu^{\otimes I}$ defined on $\mathcal{F}^{\otimes I}$. In case we are only concerned with the cardinality of $I$, we will write $\Omega^n$, where $n$ is a natural number or $\infty$. By the natural filtration on $\Omega^\mathbb{N}$ we mean the filtration $\left(\mathcal{F}_n:n=0,1,\ldots\right)$, where $\mathcal{F}_k$ is generated by the coordinate projection $\omega\mapsto \left(\omega_1,\ldots,\omega_k\right)$ and denote $\E_k=\E\left(\cdot\mid\mathcal{F}_k\right)$. In general, for a subset $A$ of the index set, $\mathcal{F}_A$ will be the sigma algebra generated by the coordinate projection $\omega\mapsto\left(\omega_i\right)_{i\in A}$ and $\E_A=\E\left(\cdot\mid\mathcal{F}_A\right)$. In more explicit terms, measurability with respect to $\mathcal{F}_A$ is equivalent to being dependent only on variables with indices belonging to $A$ and the conditional expectation operator $\E_A$ integrates away the dependence on all other variables, so that the formulas 
\beq \E_k f\p{x}= \int_{\Omega^{[k+1,\infty)}} f\p{x_1,\ldots,x_k,y_{k+1},y_{k+2},\ldots}\d \mu^{\otimes [k+1,\infty)}(y),\eeq
\beq \E_A f(x)= \int_{\Omega^{\N\setminus A}}f\p{ x_A,y_{\N\setminus A}}\d \mu^{\otimes \N\setminus A}(y)\eeq
are satisfied (with the convention that sequences indexed by $A$ and $\N\setminus A$ are merged in a natural way into a sequence indexed by $\N$). It will often be convenient to identify a function $f$ defined on $\Omega^A$ with an $\mathcal{F}_A$-measurable function $\Omega^I\ni\omega \mapsto f\p{ \p{\omega_i}_{i\in A}}$. In order to save space, we will often write $\d x$ instead of $\d \mu(x)$ whenever the measure is implied by context.  \par

\textbf{Tensor products.} Let $1\leq p<\infty$. For $f_k\in L^p\p{\Omega_k}$, we will denote by $\bigotimes_{k=1}^n f_k$ the function on $\prod_k \Omega_k$ satisfying
\beq \p{\bigotimes_k f_k}(x)=\prod_k f_k\p{x_k}.\eeq
Because of separation of variables, we have $\left\|\bigotimes_k f_k\right\|_{L^p\p{\prod_k \Omega_k}}= \prod_k \left\|f_k\right\|_{L^p\p{\Omega_k}}$. This way we actually define an injection of the algebraic tensor product $\bigotimes_k L^p\p{\Omega_k}$ into $L^p\p{\prod_k \Omega_k}$, the image of which is dense. \par
Let $X_k$ be subspaces (by a subspace we always mean a closed linear subspace) of $L^p\p{\Omega_k}$. By $\bigotimes_k X_k$ we will denote the subspace of $L^p\p{\prod_k \Omega_k}$ spanned by functions of the form $\bigotimes_k f_k$, where $f_k\in X_k$, and the norm is inherited from $L^p\p{\prod_k \Omega_k}$ (care has to be taken, as $\bigotimes_k X_k$ is not determined solely by $X_k$ as Banach spaces, but rather by the particular way they are embedded in $L^p\p{\Omega_k}$). If $T_k: X_k\to L^p\p{\Omega_k}$ are bounded operators, then we can define an operator $\bigotimes_k T_k: \bigotimes_k X_k\to L^p\p{\prod_k \Omega_k}$ by the formula
\beq \p{\bigotimes_k T_k}\p{\bigotimes_k f_k}= \bigotimes_k T_kf_k,\eeq
and easily check that the property 
\beq \left\|\bigotimes_k T_k:\bigotimes_k X_k\to L^p\p{\prod_k \Omega_k}\right\|\leq \prod_k \left\|T_k:X_k\to L^p\p{\Omega_k}\right\|\eeq
is satisfied. Indeed, $\bigotimes_k T_k= \prod_k \id_{L^p\p{\prod_{j\neq k}\Omega_j}}\otimes T_k$, and any operator of the form $\id\otimes T$ has norm bounded by $\|T\|$, because $\p{\id\otimes T}f\p{\omega_1,\omega_2}= T\p{f\p{\omega_1,\cdot}}\p{\omega_2}$. \par

\textbf{Khintchine's and related inequalities.} First, we recall the classical Khintchine-Kahane inequality. 
\begin{thm}[Khintchine for $B=\mathbb{C}$, Kahane for $B$ Banach]Let $X_1,\ldots,X_n$ be vectors in a Banach space $B$ and $r_1,\ldots,r_n$ be Rademacher variables (i.e. independent random variables, each of them attaining $\pm 1$ with probability $\frac12$). Then, for $1\leq p,q<\infty$,
\beq \p{\E\left\|\sum r_i X_i\right\|_B^p}^\frac1p \simeq_{p,q} \p{\E\left\|\sum r_i X_i\right\|_B^q}^\frac1q.\eeq\end{thm}
In particular, $\E\left|\sum r_i z_i\right|\simeq \sqrt{\sum\left|z_i\right|^2}$ for $z_i\in \mathbb{C}$. As a consequence we obtain the Marcinkiewicz-Zygmund moment inequality. 
\begin{lem}\label{marzyglp}Let $S$ be a (not necessarily probability) measure space. If $X_1,\ldots,X_n$ are independent $L^p(S)$-valued mean $0$ random variables ($1\leq p<\infty$), then \[\E\n{\sum X_i}_{L^p(S)}\simeq_p \E\n{\sqrt{\sum \left|X_i\right|^2}}_{L^p(S)}.\]\end{lem}
\textit{Proof.} If $X,Y$ are mean $0$ and independent, then $\E\|X+Y\|_{L^p}\geq \E\|X\|_{L^p}$ by taking the conditional expectation with respect to (the sigma-algebra generated by) $X$ and $\E\|X+Y\|_{L^p}\leq \E\|X\|_{L^p}+\E\|Y\|_{L^p}$ because of triangle inequality. Thus $\E\|X+Y\|_{L^p}\simeq \E\|X\|_{L^p}+\E\|Y\|_{L^p}$ and in particular $\E\|X+Y\|_{L^p}\simeq \E\|X-Y\|_{L^p}$. Hence, there is an equivalence $\E\n{\sum X_i}_{L^p}\simeq \E\n{\sum r_i X_i}_{L^p}$, where $r_i$ are Rademacher variables independent of $X_i$'s, which is true a.e. on the space where $r_i$'s are defined. Let us distinguish the expectations on spaces on which $r_i$'s and $X_i$'s are defined by denoting them respectively by $\E^{(r)}$ and $\E^{(X)}$. Then
\begin{eqnarray*} \E^{(X)}\n{\sum X_i}_{L^p}&\simeq & \E^{(r)}\E^{(X)}\n{\sum r_i X_i}_{L^p}\\
&\simeq_p & \E^{(X)}\left(\E^{(r)}\n{\sum r_i X_i}_{L^p}^p\right)^\frac1p\\ 
&=& \E^{(X)}\left(\int_S \E^{(r)}\left|\sum r_i X_i(s)\right|^p \d \mu(s)\right)^\frac1p\\
&\simeq_p& \E^{(X)}\left(\int_S \left(\sum \left|X_i(s)\right|^2\right)^\frac{p}{2} \d \mu(s)\right)^\frac1p\\
&=& \E^{(X)} \left\|\sqrt{\sum \left|X_i\right|^2}\right\|_{L^p}.
\end{eqnarray*} \par

\textbf{Vector-valued inequalities.} For a Banach space $B$, by $L^p\p{S,B}$ we denote the Bochner space of strongly measurable $B$-valued random variables equipped with the norm 
\beq \|f\|_{L^p(S,B)}=\p{\int_{s}\|f(x)\|_B^p\d\mu(s)}^\frac1p\eeq
(or, equivalently, the closed span of functions of the form $(f\otimes v)(x)= f(x)v$, where $f\in L^p(S)$ and $v\in B$, in the $L^p(S,B)$ norm). For an operator $T$ between subspaces of $L^p\p{S_1}$ and $L^p\p{S_2}$ and a linear operator $F:B_1\to B_2$ we can define $T\otimes F$ on the algebraic tensor product by $\p{T\otimes F}\p{f\otimes v}= T(f)\otimes F(v)$, but this construction does not necessarlily produce a bounded operator on the closure. The main tool for obtaining vector-vlaued extensions of inequalities will be the following lemma, which for $I_1,I_2$ being singletons is due to Marcinkiewicz and Zygmund \cite{marzyg} (in this case $\lesssim \|T\|$ can be replaced with $\leq\|T\|$). 
\begin{lem}\label{rbdd}Let $X_i\subset L^1\p{S_i,\ell^2\p{I_i}}$ for $i=1,2$, $B$ be a Hilbert space and $T:X_1\to X_2$ be bounded. Then $T\otimes \id_B: X_1\otimes B\to X_2\otimes B$, where $X_i\otimes B$ is treated as a subspace of $L^1\p{\Omega_i,\ell^2\p{I_i,B}}$, is bounded with norm $\lesssim \|T\|$. \end{lem}
\begin{proof}Without loss of generality, $B$ is finite-dimensional, say $B=\ell^2\p{J}$ for some finite $J$. Let $X_1\otimes \ell^2\p{J}\ni f=\p{f_j}_{j\in J}$, so that $f_j\in X_1$. Let also $r_j$ for $j\in J$ be Rademacher variables. Then, applying $\ell^2\p{I_2}$-valued Khintchine inequality, 
\begin{align}\left\|\p{T\otimes \id}f\right\|_{L^1\p{S_2,\ell^2\p{I_2\times J}}} =&\int_{S_2}\sqrt{\sum_{j\in J}\left\|Tf_j\p{s}\right\|_{\ell^2\p{I_2}}^2}\d \mu_2(s)\\
\simeq & \int_{S_2}\E\left\|\sum_j r_j Tf_j(s)\right\|_{\ell^2\p{I_2}}\d\mu_2(s)\\
=& \E\int_{S_2}\left\|T\p{\sum_j r_j f_j}(s)\right\|_{\ell^2\p{I_2}}\d\mu_2(s)\\
\leq& \|T\| \E\int_{S_1}\left\|\sum_j r_j f_j(s)\right\|_{\ell^2\p{I_1}}\d\mu_1(s)\\\leq& \|T\|\int_{S_1}\sqrt{\sum_j \left\|f_j(s)\right\|_{\ell^2\p{I_1}}^2}\d\mu_1(s)\\
=&\|T\| \|f\|_{L^1\p{S_1,\ell^2\p{I_1\times J}}}.
\end{align}
\end{proof}

\textbf{Interpolation.} Let us recall basic information about the real interpolation method. The standard reference is \cite{bensha}. A couple $\p{X_0,X_1}$ of Banach spaces is called compatible if $X_0$ and $X_1$ are embedded in a linear topological space. For a compatible couple $\p{X_0,X_1}$ we define the $K$-functional by the formula \[K\p{f,t;X_0,X_1}= \inf_{f=f_0+f_1} \n{f_0}_{X_0}+t\n{f_1}_{X_1}\] for $f\in X_0+X_1$. For $Y_0\subset X_0$, $Y_1\subset X_1$, we will say that the couple $\p{Y_0,Y_1}$ is $K$-closed in $\p{X_0,X_1}$ if \[ K\p{f,t;Y_0,Y_1}\lesssim K\p{f,t;X_0,X_1}\] for any $f$ in the algebraic sum $Y_0+Y_1$ (the reverse inequality holds trivially). This is equivalent to the following property: for any $f\in Y_0+Y_1$ and any decomposition $f=f_0+f_1$, where $f_0\in X_0$, $f_1\in X_1$, there exists a decomposition $f=\tilde{f_0}+\tilde{f_1}$ such that $\tilde{f_0}\in Y_0$, $\tilde{f_1}\in Y_1$, $\|\tilde{f_0}\|_{Y_0}\lesssim \n{f_0}_{X_0}$, $\|\tilde{f_1}\|_{Y_1}\lesssim \n{f_1}_{X_1}$. \par
The $K$-functional plays a crucial role in the real interpolation method. The norm in the real interpolation space $\p{X_0,X_1}_{\theta,q}$, where $0<\theta<1$ and $1\leq q<\infty$ is defined by 
\beq \|f\|_{\p{X_0,X_1}_{\theta,q}}= \p{\int_0^\infty \p{t^{-\theta}K\p{f,t;X_0,X_1}}^q\frac{\d t}{t}}^\frac1q.\eeq
Operators bounded simultaneously on $X_0$ and $X_1$ are also bounded on $\p{X_0,X_1}_{\theta,q}$. The canonical example is $\p{L^{p_0},L^{p_1}}_{\theta,p_{\theta}}= L^{p_\theta}$, where $\frac{1}{p_\theta}= \frac{1-\theta}{p_0}+\frac{\theta}{p_1}$ and, more generally, $\p{L^{p_0}\p{\ell^{q_0}},L^{p_1}\p{\ell^{q_1}}}_{\theta,p_\theta}= L^{p_\theta}\p{\ell^{q_\theta,p_\theta}}$. If $\p{Y_0,Y_1}$ is $K$-closed in $\p{X_1,X_2}$, then it is easily seen that $\p{Y_0,Y_1}_{\theta,q}= \p{Y_0+Y_1}\cap \p{X_0,X_1}_{\theta,q} $, which is particularly useful in case of couples $K$-closed in Lebesgue spaces. \par
The algebraic sum $X_0+X_1$ becomes a Banach space when equipped with the norm
\beq \|f\|_{X_0+X_1}=K\p{f,1;X_0,X_1}.\eeq
The intersection $X_0\cap X_1$ will be equipped with the norm
\beq \|f\|_{X_0\cap X_1}=\max\p{\|f\|_{X_0},\|f\|_{X_1}}.\eeq
It is easily checked that the dual space $\p{X_0+X_1}^*$ can be isometrically identified with $X_0^*\cap X_1^*$.

\textbf{Hoeffding decomposition.} In order to avoid technicalities with convergence in strong operator topology, we will work in a finite product of $\Omega$ (all the results extend automatically to $\Omega^\infty$ by density). We will see in a moment that any function $f\in L^1\left(\Omega^n\right)$ can be decomposed in a unique way as
\[ \label{eq:hoeff} f=\sum_{m=0}^n \sum_{1\leq i_1<\ldots<i_m\leq n}P_{i_1,\ldots,i_m}f,\]
where $P_{i_1,\ldots,i_m}f\p{x_1,\ldots,x_n}$ depends only on $x_{i_1},\ldots,x_{i_m}$ and is of mean $0$ with respect to each of $x_{i_1},\ldots,x_{i_m}$ (equivalently, $P_A f$ is $\mathcal{F}_A$-measurable and is orthogonal to all $\mathcal{F}_B$-measurable functions for $B\subsetneq A$). This decomposition has been studied in \cite{bourgwalsh}, \cite{kwap}. In particular, $P_{i_1,\ldots,i_m}$ are pairwise orthogonal orthogonal projections. Let \[P_m= \sum_{1\leq i_1<\ldots<i_m\leq n}P_{i_1,\ldots,i_m}\] and $\V_m$ be the range of $P_m$. It is known \cite{bourgwalsh}, \cite{kwap} that $P_m$ is bounded on $L^p\p{\Omega^n}$, $1<p<\infty$, with norm independent on $n$, but this is not true for $L^1\p{\Omega^n}$. \par
One of the possible ways to prove the existence of the above decomposition in $L^2\left(\Omega^n\right)$ is as follows. First we define the subspace
\beq \V_{\leq m}^2=\clspan \bigcup_{|A|\leq m} \left\{f\in L^2\left(\Omega^n\right): f\text{ is }\mathcal{F}_A\text{-measurable}\right\}\subset L^2\left(\Omega^n\right)\eeq 
for each $m\geq 0$. The sequence of subspaces $\V_{\leq 0}^2,\V_{\leq 1}^2,\ldots,\V_{\leq n}^2$ is increasing, so by putting \beq \V_0^2=\V_{\leq 0}^2, \quad \V_m^2=\V_{\leq m}^2\cap \p{\V_{\leq m-1}^2}^{\perp}\eeq we obtain a decomposition \beq L^2\left(\Omega^n\right)= \bigoplus_{m=0}^{n} \V_m^2\eeq into an orthogonal direct sum of $\V_m^2$. We will denote the orthogonal projection onto $\V_m^2$ by $P_m$ and the closure of $U^2_m$ equipped with $L^p$ norm by $U^p_m$. \par
A more explicit formula for $P_m$ can be obtained. For $A\subset [1,n]$, let 
\beq\label{eq:pdef} P_A= \left(\id-\E\right)^{\otimes A}\otimes \E^{\otimes [1,n]\setminus A},\eeq where $\id$ and $\E$ are understood to act on $L^2(\Omega)$, and let $\V_A^2$ be the range of the projection $P_A$. It is easy to see that 
\beq \label{eq:iddecomp}\E_A= (\id- \E+\E)^{\otimes A}\otimes \E^{\otimes [1,n]\setminus A}=  \sum_{B\subset A}\left(\id-\E\right)^{\otimes B}\otimes\E^{\otimes [1,n]\setminus B}\eeq 
and, since the subspaces $\V_B^2$ are mutually orthogonal, 
\beq  \label{eq:idoplus}L^2\left(\Omega^n,\mathcal{F}_A\right)= \bigoplus_{B\subset A} \V_B^2.\eeq
Moreover 
\begin{eqnarray} \V_{\leq m}^2&=& \clspan \bigcup_{|A|\leq m}L^2\left(\Omega^n,\mathcal{F}_A\right)\\ &=&  \clspan \bigcup_{|A|\leq m}  \bigoplus_{B\subset A} \V_B^2\\ &=&  \bigoplus_{|B|\leq m} \V_B^2\end{eqnarray} and consequently \beq\label{eq:pdecomp} \V_m^2= \bigoplus_{|B|=m} \V_B^2, \quad P_m=\sum_{|B|=m}P_B.\eeq 

\textbf{Decoupling inequalities.} We are going to present a special case of a theorem of J. Zinn \cite{zinn}, which will be one of the most important tools. \par
\begin{thm}[Zinn]\label{tildemain}Let $X_1,\ldots,X_n,Y_1,\ldots,Y_n$ be independent random variables such that $X_k$ and $Y_k$ have the same distribution for any $k$. Let also $\varphi_k$ be a nonnegative Borel function on $\R^k$ and $0<q\leq 1$. Then
\beq \E\p{\sum_k \varphi_k\p{X_1,\ldots,X_k}}^q\simeq \E\p{\sum_k \varphi_k\p{X_1,\ldots,X_{k-1},Y_k}}^q.\eeq
\end{thm}

\begin{cor}\label{multizinn}For all $i=\p{i_1,\ldots,i_m}$ such that $i_1<\ldots<i_m$, let $f_i$ be an $\mathcal{F}_{\left\{i_1,\ldots,i_m\right\}}$-measurable nonnegative function on $\Omega^\N$. Then, treating each $f_i$ as a function on $ \Omega^{\left\{i_1,\ldots,i_m\right\}}$,
\beq
\label{eq:multidec}\int_{\Omega^\N}\p{\sum_i f_i\p{x_{i_1},\ldots,x_{i_m}}}^q\d x\simeq_m  \int_{\p{\Omega^\N}^m}\p{\sum_i f_i\p{y^{(1)}_{i_1},\ldots,y^{(m)}_{i_m}}}^q\d y^{(1,\ldots,m)},
\eeq
where $y^{(1)},\ldots,y^{(m)}$ are variables in $\Omega^\N$ and $0<q\leq 1$. 
\end{cor}
\begin{proof} Let us fix $k\in\{1,\ldots,m\}$ and for each $j\in \N$ define a function $\varphi_j$ on $\Omega^{\left[1,j\right]}\times \p{\Omega^\N}^{m-k}$ by the formula
\begin{align} &\varphi_j \p{x_{\leq j},y^{(k+1)},\ldots,y^{(m)}}=\\
\nonumber&\sum_{\substack{i_1<\ldots<i_{k-1}<\\ j<i_{k+1}<\ldots<i_{m}}} f_{i_1,\ldots,i_{k-1},j,i_{k+1},\ldots,i_{m}}\p{x_{i_1},\ldots,x_{i_{k-1}},x_j,y^{(k+1)}_{i_{k+1}},\ldots,y^{(m)}_{i_m}}.\end{align}
Then, for fixed $y^{(> k)}=\p{ y^{(k+1)},\ldots,y^{(m)}}\in\p{\Omega^\N}^{m-k}$,  
\begin{align}&\int_{\Omega^\N}\p{\sum_{i_1<\ldots<i_m} f_i\p{x_{i_1},\ldots,x_{i_{k}},y^{(k+1)}_{i_{k+1}},\ldots,y^{(m)}_{i_m}}  }^q\d x \\
=& \int_{\Omega^\N}\p{\sum_{j\in \N} \varphi_j \p{x_{\leq j},y^{(k+1)},\ldots,y^{(m)}} }^q\d x \\
\label{eq:appdec2phi}\simeq& \int_{\Omega^\N}\int_{\Omega^\N}\p{\sum_{j\in \N} \varphi_j \p{x_{<j},y^{(k)}_{j},y^{(k+1)},\ldots,y^{(m)}}  }^q\d x\d y^{(k)}\\
=& \int_{\Omega^\N}\int_{\Omega^\N}\p{\sum_{i_1<\ldots<i_m} f_i\p{x_{i_1},\ldots,x_{i_{k-1}},y^{(k)}_{i_k},y^{(k+1)}_{i_{k+1}},\ldots,y^{(m)}_{i_m}}  }^q\d x\d y^{(k)}.
\end{align}
Here, $i_k$ plays the role of $j$ and \eqref{eq:appdec2phi} is an application of Theorem \ref{tildemain} to functions $\varphi_j$. Integrating the resulting inequality with respect to $y^{(>k)}$, we get 
\begin{align} &\int_{\p{\Omega^\N}^{m-k}}\int_{\Omega^\N}\p{\sum_{i_1<\ldots<i_m} f_i\p{x_{i_1},\ldots,x_{i_{k}},y^{(k+1)}_{i_{k+1}},\ldots,y^{(m)}_{i_m}}  }^q\d x\d y^{(\geq k+1)}\\
\nonumber \simeq & \int_{\p{\Omega^\N}^{m-k+1}}\int_{\Omega^\N}\p{\sum_{i_1<\ldots<i_m} f_i\p{x_{i_1},\ldots,x_{i_{k-1}},y^{(k)}_{i_k},y^{(k+1)}_{i_{k+1}},\ldots,y^{(m)}_{i_m}}  }^q\d x\d y^{(\geq k)},\end{align}
which by induction from $k=m$ to $k=1$ proves \eqref{eq:multidec}.  \end{proof}

\section{Weighted inequalities}\label{secweighted}

If $\left(\Omega_i,\mu_i,\mathcal{F}_i\right)$ are probability spaces, we can form a (no longer probability) measure space $\left(\coprod_i\Omega_i, \coprod_i \mu_i,\coprod_i \mathcal{F}_i\right)$ by considering the disjoint union of sets $\Omega_i$ with a measure $\coprod_i \mu_i\left(\bigcup A_i\right)=\sum_i \mu_i\left(A_i\right)$ for $A_i\in\mathcal{F}_i$. We will often write $\d x$ instead of $\d \mu(x)$ if the choice of measure is clear. \par
Our main motivation is the following special case of a theorem due to Johnson and Schechtman. 
\begin{thm}[Johnson, Schechtman \cite{JSch}]\label{JSch} Let $f_i\in L^1(\Omega)$ for $i=1,\ldots,n$ and $1\leq p<\infty$. Then \beq \label{eq:JSchineq}\int_{\Omega^n}\left(\sum_i \left|f_i\left(x_i\right)\right|^p\right)^\frac1p \d x\simeq_p \inf_{\coprod_i f_i=g+h}\left\|g\right\|_{L^1\left(\coprod_i\Omega\right)} + \left\|h\right\|_{L^p\left(\coprod_i\Omega\right)}.\eeq In more explicit terms, the inequality `$\gtrsim$' means that if \[ \int_{\Omega^n}\left(\sum_i \left|f_i\left(x_i\right)\right|^p\right)^\frac1p \d x\leq 1,\] then there exists a decomposition $f_i=g_i+h_i$ such that \[\sum_i \int_{\Omega} \left|g_i\left(x_i\right)\right|\d x_i\lesssim 1,\quad \sum_i \int_{\Omega} \left|h_i\left(x_i\right)\right|^2\d x_i\lesssim 1.\] Moreover, the decomposition can be chosen so that for any $i$ the supports of $g_i$ and $h_i$ are disjoint.\end{thm}
This can be use to expresses $L^p$ norm, $1\leq p\leq 2$ (the case $\infty>p\geq 2$ is handled by Rosenthal inequality), on $\V_1^p$ as a rearrangement invariant norm on the disjoint union $\coprod_i \Omega$ in the following way. For $f\in \V_1^p$, we have 
\beq f(x)=\sum_i f_i\p{x_i}\eeq 
for some $f_i$ of mean $0$ and consequently 
\beq \|f\|_{L^p\p{\Omega^n}}\simeq_p \p{\int_{\Omega^n}\p{\sum_i \left|f_i\p{x_i}\right|^2}^\frac{p}{2}\d x}^\frac{1}{p}\eeq 
by Marcinkiewicz-Zygmund inequality. Applying Theorem \ref{JSch} with expontent $\frac{2}{p}$ to $\left|f_i\right|^{p}$, we get an equivalent form of Theorem \ref{JSch}. 
\begin{cor}\label{1varmzmz}For $1\leq p\leq 2$ and $f_i\in L^p(\Omega)$ of mean zero,
\beq \p{\int_{\Omega^n} \left|\sum_i f_i\p{x_i}\right|^p}^\frac{1}{p}\simeq_p \inf_{\bigsqcup_i f_i =g+h} \|g\|_{L^2\p{\bigsqcup_i\Omega}}+ \|h\|_{L^p\p{\bigsqcup_i\Omega}}.\eeq
\end{cor}
The building blocks $g_i$, $h_i$ of $g,h$ in the above can be also chosen to be of mean $0$, because 
\beq f_i= (\id-\E)f_i= (\id-\E)g_i+ (\id-\E)h_i.\eeq \par
In this section we will develop a weighted version of Theorem \ref{JSch} (in particular producing a new proof with constant independent on $p$), which will be Theorem \ref{premainlemma}, which for $J$ a singleton, $w\equiv1$, $\kappa=1$ and $\varepsilon=0$ gives a constant $\frac{1}{2}$. \par
We start with a calculation lemma.
\begin{lem}\label{calclemma}
Let $\phi:\R^\N\to \R$ be a function differentiable outside of $0$ and satisfying $\phi(cx)=c\phi(x)$ for $c>0$. Then
\beq \sum_{k=1}^N x_k \frac{\partial \phi(x)}{\partial x_k}=\phi(x).\eeq
\end{lem}
\begin{proof}
Let $|\cdot|$ denote the Euclidean norm. By direct calculation we verify that
\beq \frac{\partial}{\partial x_k} \frac{x_j}{|x|}= \delta_{jk}|x|^{-1}-x_jx_k|x|^{-3}.\eeq
Therefore
\beq \sum_k x_k \frac{\partial}{\partial x_k} \frac{x}{|x|}=0.\eeq
Thus
\begin{eqnarray}\sum_{k=1}^N x_k \frac{\partial \phi(x)}{\partial x_k}&=& \sum_{k=1}^N x_k \frac{\partial }{\partial x_k}\p{|x|\phi\p{\frac{x}{|x|}}}\\
&=& \sum_k x_k\p{ \frac{x_k}{|x|}\phi\p{\frac{x}{|x|}}+ |x|\left\langle \nabla\phi\p{\frac{x}{|x|}}, \frac{\partial}{\partial x_k}\frac{x}{|x|}\right\rangle}\\
&=&\sum_k \frac{x_k^2}{|x|}\phi\p{\frac{x}{|x|}}+ |x|\left\langle \nabla\phi\p{\frac{x}{|x|}}, \sum_k x_k\frac{\partial}{\partial x_k}\frac{x}{|x|}\right\rangle\\
&=& \phi(x).
\end{eqnarray}
\end{proof}
The following inequality is useful for obtaining lower bounds for subspaces of vector valued $L^1$. 
\begin{lem}\label{gradlemma} Let $N\in \N$, $\p{\Omega,\mathcal{F},\mu}$ be a finite probability space, $V_k$ be subspaces of $L^2\p{\Omega}$, $V$ be the subspace of $L^2\p{\Omega,\ell^2_N}$ consisting of sequences $\p{f_{n}}_{n=1}^N$ of functions such that $f_n\in V_n$, $P_{V_k}$, $P_V$ be orthogonal projections onto $V_k$, $V$ respectively, $\|\cdot\|_X$ be a random norm on $\R^N$ differentiable outside of $0$, $\|\cdot\|_Y$ be a norm on $V$, $Y^*$ be the dual norm on $V$ in the sense of the usual pairing. Then for any constants $C>0$, $q> 1$, the following are equivalent:\\
(i) for any $\p{f_1,\ldots,f_N}\in V$,
\beq \label{eq:lemgradi}\E \left\|\p{f_n}_{n=1}^N\right\|_X^q\geq C\left\|\p{f_n}_{n=1}^N\right\|_Y^q\eeq
(ii) for any $\p{f_1,\ldots,f_N}\in V$ not identically zero,
\beq \label{eq:lemgradii}\left\| P_V\p{\|f\|_X^{q-1} \nabla\| \cdot \|_X (f)}\right\|_{Y^*}\geq C \|f\|_Y^{q-1},\eeq
where $\nabla\|\cdot\|_X$ is extended to be equal to $0$ in $0$. Moreover, if for any $\p{f_1,\ldots,f_N}\in V$ not identically zero,
\beq \label{eq:weirdcond}\left\| P_{V\cap F}\p{\nabla\|\cdot\|_X(f)}\right\|_{Y^*}\geq C\left\|P_{V\cap F}\right\|_{Y^*\to Y^*},\eeq
where $F=\left\{\varphi\in V: \supp \varphi_i\subset \supp f_i\right\}$, then 
\beq \label{eq:l1xgeqcy}\E\|f\|_X\geq C\|f\|_Y\eeq 
for all $f\in V$. 
\end{lem}
\begin{proof}
Let us start with $q>1$. The implication (i)$\implies$(ii) is true for each $f$ separately. Indeed, if \eqref{eq:lemgradi} holds, then by self-adjointness of $P_{V_n}$ and Lemma \ref{calclemma} applied to $\|\cdot\|_X$ pointwise to $f$, 
\begin{align} \lefteqn{ \left\| P_V\p{\|f\|_X^{q-1} \nabla\| \cdot \|_X (f)} \right\|_{Y^*} \|f\|_Y}\\
\geq & \sum_n \E\p{ f_n P_{V_n}\p{ \|f\|_X^{q-1}  \partial_n\| \cdot \|_X (f)}}\\
= &  \E\p{ \sum_n f_n  \partial_n\|\cdot\|_X \|f\|_X^{q-1} }\\
=& \E\left\|f\right\|_X^q\\
\geq & C \left\|f\right\|_Y^q.\end{align}
We will prove the implication (ii)$\implies $(i) now. Let us assume $q>1$ first. We are in a fully finite-dimensional setting. It is enough to prove 
\beq \E \left\|\p{f_n}_{n=1}^N\right\|_X^q\geq C \left|\sum_n \E\p{f_n g_n}\right|^q\eeq
for any $g=\p{g_1,\ldots,g_N}\in V$ such that 
\beq \label{eq:lemnormg} \|g\|_{Y^*}=1.\eeq
By homogeneity, we can set
\beq \label{eq:lemnormf}\sum_n \E\p{f_n g_n}=1.\eeq
Since $\p{\E\|\cdot\|_X^q}^\frac{1}{q}$ is a norm on $V$, under the constraint \eqref{eq:lemnormf} it is a convex function going to $\infty$ in infinity, so it attains a minimum. Let $f$ be a minimizer of $\E\|\cdot\|_X^q$. Suppose for a moment $q>1$. For any $h\in V$ such that $\langle h,g\rangle=0$ we have
\begin{eqnarray} 0&=& \left.\frac{\d}{\d t}\right|_{t=0} \E\left\|f+th\right\|_X^q\\
&=& q\E \left\langle \|f\|_X^{q-1}\nabla\|\cdot\|_X(f),h\right\rangle\\
&=& q\left\langle P_V \p{\|f\|_X^{q-1}\nabla\|\cdot\|_X(f)},h\right\rangle.
\end{eqnarray}
In other words, 
\beq \label{eq:kercontain}\ker g \subset \ker P_V \p{\|f\|_X^{q-1}\nabla\|\cdot\|_X(f)}\eeq when both are treated as functionals on $V$. 	Therefore
\beq P_V \p{\|f\|_X^{q-1}\nabla\|\cdot\|_X(f)}=\lambda g\eeq
for some $\lambda \in \mathbb{R}$, which by \eqref{eq:lemgradi} and \eqref{eq:lemnormg} gives 
\beq C \|f\|_Y^{q-1}\leq \left\| P_V \p{\|f\|_X^{q-1}\nabla\|\cdot\|_X(f)}\right\|_{Y^*}= |\lambda| \|g\|_{Y^*}= |\lambda|.\eeq
Ultimately, utilising Lemma \ref{calclemma} again,
\begin{eqnarray} \E\|f\|_X^q &=& \E \|f\|_X^{q-1}\|f\|_X\\
&=& \E\|f\|_X^{q-1}\sum_n f_n \partial_n \|\cdot\|_X (f)\\
&=& \E\sum_n f_n P_{V_n}\p{\|f\|_X^{q-1}\partial_n \|\cdot\|_X (f)}\\
&=& \left\langle f, P_V \p{\|f\|_X^{q-1}\nabla\|\cdot\|_X(f)}\right\rangle  \\
&=&\lambda \langle f,g\rangle =\lambda.\end{eqnarray}
Therefore $\lambda\geq 0$ and thus $\lambda=|\lambda|\geq C \|f\|_Y^{q-1}$ as desired. \par
Now we will prove that \eqref{eq:weirdcond} is sufficient for \eqref{eq:l1xgeqcy}. We will proceed in an analogous manner, but we have to take care of nondifferentiability of $\|\cdot\|_X$ in $0$. As previously, we pick $g\in V$ such that
\beq \|g\|_{Y^*}=1\eeq
and normalize $f$ to satisfy
\beq\label{eq:fg1oncemore}\langle f,g\rangle=1.\eeq
If $f$ is a minimizer of $\E\|f\|_X$ given \eqref{eq:fg1oncemore}, then for any $h\in V\cap F$, the function $t\mapsto \|f+th\|_X$ is in each point of $\Omega$ either differentiable in $0$ (when $f\neq 0$) or identically $0$ (when $f=0$ and consequently $h=0$). Thus, for any $h\in V\cap F$ such that $\langle h,g\rangle=0$,  
\begin{eqnarray} 0&=& \left.\frac{\d}{\d t}\right|_{t=0} \E\left\|f+th\right\|_X\\
&=& \E \left\langle \nabla\|\cdot\|_X(f),h\right\rangle\\
&=& \left\langle P_{V\cap F} \p{\nabla\|\cdot\|_X(f)},h\right\rangle.
\end{eqnarray}
Therefore $\ker g\subset \ker P_{V\cap F} \p{\nabla\|\cdot\|_X(f)}$ as functionals on $V\cap F$, so
\beq P_{V\cap F} \p{\nabla\|\cdot\|_X(f)}=\lambda P_{V\cap F}g,\eeq
in particular
\beq \left\|P_{V\cap F} \p{\nabla\|\cdot\|_X(f)}\right\|_{Y^*}=\left|\lambda\right| \left\|P_{V\cap F}g\right\|_{Y^*} \leq \left|\lambda\right| \left\|P_{V\cap F}\right\|_{Y^*\to Y^*},\eeq
so $\left|\lambda\right|\geq C$. Thus
\begin{eqnarray} \E\|f\|_X &=&\E \left\langle \nabla\|\cdot\|_X (f),f\right\rangle\\
&=& \left\langle P_{V\cap F} \nabla\|\cdot\|_X (f) ,f\right\rangle\\
&=& \lambda \left\langle P_{V\cap F} g ,f\right\rangle\\
&=& \lambda = \left|\lambda\right|\geq C.
\end{eqnarray}

\end{proof}
Now, we are ready for the main result of this section. The parameter $\varepsilon$ is for technical reasons and we will make most use of the case $\varepsilon=0$ and $\{0,1\}$-valued $w$, in which case the inequality is true with the constant $\frac{1}{2}\kappa^2$. Also, as will be noted in the proof of Theorem \ref{JSch}, the decomposition in the interpolation norm on the right hand side may always be chosen to be of disjoint supports at the cost of constant 2. Sometimes we will use a weaker version of Theorem \ref{premainlemma} with the norm on the right hand side replaced by a smaller norm $L^1\p{\Omega,\ell^p(I\times J)}$ of the sequence $\p{  \mathbbm{1}_{\left\{\E_i \p{w_{i,j}\vee\varepsilon}\geq \kappa\right\}}f_{i,j} }_{i,j} $. 
\begin{thm}\label{premainlemma}Let $\Omega$ be a finite probability space, $I,J$ be finite sets, $\p{\mathcal{F}_i}_{i\in I}$ be an independent family of sigma-algebras, $f_{i,j}$ be $\mathcal{F}_i$-measurable, $w_{i,j}$ be $[0,1]$-valued functions on $\Omega$ and $p\geq 1$, $0\leq\varepsilon<\kappa\leq 1$. Then
\beq \E\p{\sum_{i,j} \left|\p{w_{i,j}\vee\varepsilon}f_{i,j}\right|^p}^\frac{1}{p}\geq C_{p,\kappa}\left\| \bigsqcup_{i} \p{ \mathbbm{1}_{\left\{\E_i \p{w_{i,j}\vee\varepsilon}\geq \kappa\right\}}f_{i,j} }_{j\in J}\right\|_{\p{L^1+L^p}\p{\bigsqcup_i \p{\Omega,\mathcal{F}_i,\mu},\ell^p(J)}},\eeq
where $C_{p,\kappa}=\kappa^p 2^{-\frac{1}{p'}}$. Moreover, if $w_{i,j}$ are $\{0,1\}$-valued, the inequality holds with constant $C_{p,\kappa}=\p{\kappa-\varepsilon}^{2-\frac{1}{p}} 2^{-\frac{1}{p'}}$. 
 \end{thm}
\begin{proof}Every time we encounter a fraction that can have $0$ denominator, it will also have $0$ numerator and it is to be treated as $0$. The notation $t^r$ stands for $|t|^r \sgn t$ so that $\frac{\d |t|^p}{\d t}= pt^{p-1}$. For $p=1$ the inequality is elementary, so we assume $p>1$. We may also assume $\varepsilon>0$ and get $\varepsilon=0$ by taking limits. \par
Let $A_{i,j}= \left\{\E_i \p{w_{i,j}\vee\varepsilon}\geq \kappa\right\}$ and $V_{i,j}= \left\{\varphi\in L^1\p{\mathcal{F}_i}:\supp\varphi\subset A_{i,j}\right\}$. We can without loss of generality assume $f_{i,j}\in V_{i,j}$, because otherwise we replace $f_{i,j}$ by $\mathbbm{1}_{A_{i,j}}f_{i,j}$. We are now in the setting of Lemma \ref{gradlemma}, with $\|f\|_X= \left\|\p{\p{w_{i,j}\vee\varepsilon}f_{i,j}}_{i,j}\right\|_{\ell^p(I\times J)}$ and $Y=\p{L^1+L^p}\p{\bigsqcup_i \p{\Omega,\mathcal{F}_i,\mu},\ell^p(J)}$. For a given $f\in V$, we have 
\beq V\cap F=\left\{\varphi:\varphi_{i,j}\text{ is }\mathcal{F}_{i}\text{-measurable and }\supp\varphi_{i,j}\subset \supp f_{i,j}\right\},\eeq
therefore the projection $P_{V\cap F}$ is just
\beq P_{V\cap F}\p{\varphi}=\p{ \E_i \mathbbm{1}_{\supp f_{i,j}} \varphi_{i,j}}_{i,j},\eeq
because the projections onto adapted sequences and onto sequences supported on $\supp f_{i,j}$ are respectively $\E_i$ and multiplication by $\mathbbm{1}_{\supp f_{i,j}}$ applied coordinatewise (in particular they commute). Moreover, $P_{V\cap F}$ is a contraction on $Y^*= \p{L^\infty\cap L^{p'}}\p{\bigsqcup_i \p{\Omega,\mathcal{F}_i,\mu},\ell^{p'}(J)}$. One easily calculates
\beq \nabla\|\cdot\|_X(f)= \p{\frac{ \p{w_{i,j}\vee\varepsilon}^p f_{i,j}^{p-1}}{\left\| \p{w\vee\varepsilon} f\right\|_{\ell^p(I\times J)}^{p-1}}}_{i,j}.\eeq 
By Lemma \ref{gradlemma}, it is enough to prove
\beq \label{eq:wijstillthere}\left\|\bigsqcup_i\p{\E_i \frac{ \p{w_{i,j}\vee\varepsilon}^p f_{i,j}^{p-1}}{\left\| \p{w\vee\varepsilon}f\right\|_{\ell^p(I\times J)}^{p-1}}}_{j\in J}\right\|_{\p{L^\infty\cap L^{p'}}\p{\bigsqcup_i \p{\Omega,\mathcal{F}_i,\mu},\ell^{p'}(J)}}\geq C_{p,\kappa}, \eeq
because $\frac{ \p{w_{i,j}\vee\varepsilon}^p f_{i,j}^{p-1}}{\left\| \p{w\vee\varepsilon}f\right\|_{\ell^p(I\times J)}^{p-1}}$ is already supported on $\supp f_{i,j}$. Let us fix $i,j$ for a moment. On each atom of $\mathcal{F}_i$ contained in $\supp f_{i,j}\subset A_{i,j}$ we have by H\"{o}lder the inequality
\beq \left( \p{ \E_i \frac{ \p{w_{i,j}\vee\varepsilon}^p }{ \|\p{w\vee\varepsilon}f\|_{\ell^p}^{p-1} }} \p{ \E_i \| \p{w\vee\varepsilon}f \|_{\ell^p}^p}^{\frac{p-1}{p}} \right)^{\frac{1}{1+ \frac{p-1}{p}}} \geq \E_i w_{i,j}^{\frac{p^2}{2p-1}}.\eeq
By rearranging the terms and multiplying by $\left|f_{i,j}\right|^{p-1}$, 
\beq \label{eq:aowiufevn}\left|\E_i \frac{ \p{w_{i,j}\vee\varepsilon}^p f_{i,j}^{p-1} }{  \left\|\p{w\vee\varepsilon}f\right\|_{\ell^p}^{p-1}  }\right|\geq \frac{\left|f_{i,j}\right|^{p-1}}{ \p{\E_i\|f\|_{\ell^p}^p}^{\frac{p-1}{p}}} \p{\E_i \p{w_{i,j}\vee\varepsilon}^\frac{p^2}{2p-1}}^{\frac{2p-1}{p}}.\eeq
This inequality has been proved on $\supp f_{i,j}$, but outside of it both sides are $0$, so it is true everywhere. Moreover, on $A_{i,j}$ we have
\beq \p{\E_i \p{w_{i,j}\vee\varepsilon}^\frac{p^2}{2p-1}}^{\frac{2p-1}{p}}= \p{\p{\E_i \p{w_{i,j}\vee\varepsilon}^\frac{p^2}{2p-1}}^{\frac{2p-1}{p^2}}}^p\geq \kappa^p.\eeq
If additionally $w_{i,j}$ is $\{0,1\}$-valued, 
\begin{eqnarray} \p{\E_i \p{w_{i,j}\vee\varepsilon}^\frac{p^2}{2p-1}}^{\frac{2p-1}{p}} &\geq& \p{\E_i w_{i,j}}^{\frac{2p-1}{p}}\\
&\geq & \p{-\varepsilon+\E_i \p{w_{i,j}\vee\varepsilon }}^{\frac{2p-1}{p}}\\
&\geq & \p{\kappa-\varepsilon}^{2-\frac{1}{p}}.
\end{eqnarray}
Plugging bounds for the factor involving $w$ to \eqref{eq:aowiufevn} and then to \eqref{eq:wijstillthere}, we see that it remains to prove
\beq \label{eq:glupieellj}\left\| \bigsqcup_i \p{\frac{\left|f_{i,j}\right|^{p-1}}{ \p{\E_i\|f\|_{\ell^p(I\times J)}^p}^{\frac{p-1}{p}}} }_{j\in J}\right\|_{\p{L^\infty\cap L^{p'}}\p{\ell^{p'}(J)}}\geq 2^{-\frac{1}{p'}}.\eeq
We have
\beq \left\|  \p{ \left|x_j\right|^{p-1} }_{j\in J}\right\|_{\ell^{p'}(J)}=\left\| \p{x_j}_{j\in J}\right\|_{\ell^p(J)}^{p-1},\eeq
so in \eqref{eq:glupieellj} we can replace $J$ with a singleton and write $f_i$ in the place of $\left\|\p{f_{i,j}}_{j\in J}\right\|_{\ell^p(J)}$, which transforms the inequality into
\beq \left\| \bigsqcup_i \p{ \frac{f_{i} }{ \p{\E_i\|f\|_{\ell^p}^p}^{\frac{1}{p}}} }^{p-1}\right\|_{L^\infty\cap L^{p'} }\geq 2^{-\frac{1}{p'}}\eeq
for nonnegative and $\mathcal{F}_i$-measurable $f_i$. Raising both sides to the power $\frac{1}{p-1}$, we end up with
\beq \label{eq:almstwghtdlem} \left\| \bigsqcup_i \frac{f_{i} }{ \p{\E_i\|f\|_{\ell^p}^p}^{\frac{1}{p}}} \right\|_{L^\infty\cap L^p }\geq 2^{-\frac{1}{p}}.\eeq
Suppose 
\beq \left\| \bigsqcup_i \frac{f_{i} }{ \p{\E_i\|f\|_{\ell^p}^p}^{\frac{1}{p}}} \right\|_{L^\infty }<2^{-\frac1p}.\eeq
Then for all $i$,
\beq f_i^p <\frac{1}{2} \E_i \|f\|_{\ell^p}^p = \frac{1}{2}\p{ f_i^p+ \E\sum_{k\neq i}f_k^p},\eeq
so
\beq f_i^p <\E\sum_{k\neq i}f_k^p.\eeq
Ultimately,
\beq \left\| \bigsqcup_i \frac{f_{i} }{ \p{\E_i\|f\|_{\ell^p}^p}^{\frac{1}{p}}} \right\|_{L^p}^p = \sum_i \E \frac{f_i^p}{f_i^p+ \E\sum_{k\neq i}f_k^p}\geq \sum_i \E\frac{f_i^p}{2\E\sum_k f_k^p}=\frac{1}{2}.\eeq
\end{proof}

As a byproduct, we obtain another proof of Theorem \ref{JSch}, which we present for the sake of completeness.\par
\textit{Proof of Theorem \ref{JSch}.} For a fixed $n$, both sides of \eqref{eq:JSchineq} are norms of the vector-valued function $\left(f_{i}\right)_{i=1,\ldots,n}$, dominated by $\sum_{i}\left\|f_{i}\right\|_{L^1\left(\Omega\right)}$. We can assume that $f_i$'s attain finitely many values due to density of such $\p{f_i}_{i\in I}$ in $\bigoplus_{i}L^1\left(\Omega\right)$. Let $\mathcal{F}$ be the sigma-algebra generated by $f_1,\ldots,f_n$, or equivalently by the partition of $\Omega$ into intersections of their level sets. We can restrict the infimum on the right side of \eqref{eq:JSchineq} to $g_i,h_i$ being $\mathcal{F}$-measurable, because for any decomposition $f_i= g_i+h_i$, the decomposition $f_i= \E\p{g_i\mid\mathcal{F}}+ \E\p{h_i\mid\mathcal{F}}$ produces a smaller result. Thus we can think of $\Omega$ being a finite space, the atoms being elements of the partition. Since both sides of \eqref{eq:JSchineq} are continuous in $\mu$, we can assume that the atoms have measures being multiples of $\frac{1}{N}$ for some $N\in\mathbb{N}$. This enables us to split each of them into atoms of size $\frac{1}{N}$, getting a new probability space $\tilde\Omega$ equipped with the normalized counting measure and a new sigma-algebra $\tilde{\mathcal{F}}$, with respect to which $f_i$'s are still measurable. By the same argument as before, we can extend the infimum to decompositions measurable with respect to $\tilde{\mathcal{F}}$. Therefore, we can assume without loss of generality that $\Omega $ is finite with normalized counting measure.\par
The $\lesssim$ inequality of \eqref{eq:JSchineq} is elementary and holds with constant $1$, because for any decomposition $f_i=g_i+h_i$ we have \begin{eqnarray*} \int_{\Omega^n}\p{\sum_i \left|f_i\p{x_i}\right|^p}^\frac1p\d x&\leq & \int_{\Omega^n}\p{\sum_i \left|g_i\p{x_i}\right|^p}^\frac1p\d x + \int_{\Omega^n}\p{\sum_i \left|h_i\p{x_i}\right|^p}^\frac1p\d x\\ &\leq& \int_{\Omega^n}\sum_i\left|g_i\p{x_i}\right|\d x + \p{\int_{\Omega^n}\sum_i \left|h_i\p{x_i}\right|^p\d x}^\frac1p\\ &=& \sum_i\int_{\Omega}\left|g_i\p{x_i}\right|\d x_i + \p{\sum_i\int_{\Omega} \left|h_i\p{x_i}\right|^p\d x_i}^\frac1p.\end{eqnarray*}
The other inequality, is precisely Lemma \ref{premainlemma} with $J$ being a singleton, $I=\{1,\ldots,n\}$, $w\equiv 1$, $\kappa=1$ and $\varepsilon=0$. It remains to prove that we can choose the desired decomposition so that the summands have disjoint supports. Without loss of generality, $f=\bigsqcup_i f_i\geq 0$. Let $f=g+h$ be some decomposition satisfying $\int_{\bigsqcup_i \Omega} g+ \p{\int_{\bigsqcup_i \Omega} h^p}^\frac{1}{p}\leq 1$. Then for any $x\in\bigsqcup\Omega$ we have either $g(x)\geq \frac12 f(x)$ or $h(x)\geq \frac12 f(x)$. In the former case we put $\tilde{g}(x)=f(x)$, $\tilde{h}(x)=0$ and in the latter $\tilde{g}(x)=0$, $\tilde{h}(x)=f(x)$, choosing arbitrarily if $g(x)=h(x)$. This way we have $\tilde{g}+\tilde{h}=f$ and $\|\tilde{g}\|_{L^1}\leq 2\|g\|_{L^1}$ and $\|\tilde{h}\|_{L^p}\leq 2\|h\|_{L^p}$. \par
In the original version of Theorem \ref{JSch}, the decomposition was defined in terms of decreasing rearrangement of $f$. There is also another way of constructing a more explicit decomposition. 
Define a function $\Phi:\mathbb{R}_+\to\mathbb{R}_+$ by 
\[\Phi(t)= \left\{\begin{matrix}t^p&\text{ if }&0\leq t\leq 1\\ pt-(p-1)&\text{ if }&t\geq 1\end{matrix}\right.\] and take 
\beq \label{eq:sliceatlevel}\alpha=\mathbbm{1}_{\{f\geq 1\}}f, \quad \beta= \mathbbm{1}_{\{f<1\}}f.\eeq
By convexity of $\Phi$ and the inequality $\Phi(t)\leq \min\p{t^p, pt}$, 
\begin{eqnarray*} \int_{\coprod_i\Omega}\p{\alpha +\beta^p}&\leq& \int_{f\geq 1}((pf-(p-1))+\int_{f<1} f^p\\ &=&\int_{\coprod_i\Omega}\Phi(f)\\ &=&\int_{\coprod_i\Omega}\Phi\p{g+h}\\&\leq&\frac12\int_{\coprod_i\Omega}\p{\Phi(2g)+\Phi(2h)}\\ &\lesssim_p& \int_{\coprod_i\Omega}g+\int_{\coprod_i\Omega}h^p\\ &\leq& 2 \end{eqnarray*}
as desired. \qed

\section{Decomposition theorem for  $\V^p_m\p{\Omega^\infty}$}\label{secdecomp}
We would like to extend Theorem \ref{premainlemma} to moments of $m$-th order $U$-statistics, and as a result extend Corollary \ref{1varmzmz} to treat $f\in \sum_{m=0}^M \V_m^p= \V_{\leq M}^p$. Let us note the following theorem due to Bourgain \cite{bourgwalsh}, which generalizes Marcinkiewicz-Zygmund inequality and in particular implies that $U^p_{m}$ is complemented in $U^p_{\leq M}$ for $m\leq M$. 
\begin{thm}[Bourgain]\label{bourgsqfn}Let $f\in \V_{\leq M}^p$. Then 
\[ \|f\|_{L^p\p{\Omega^n}}\simeq_{M,p} \left\|\sqrt{\sum_{|A|\leq M} \left|P_A f\right|^2}\right\|_{L^p\p{\Omega^n}}.\] \end{thm}

\begin{cor}\label{sqfndecoup}Let $f\in \V_{\leq M}^p$ and $P_{{i_1},\ldots,i_m}f\p{x}=f_{i_1,\ldots,i_m}\p{x_{i_1},\ldots,x_{i_m}}$. Then
\[ \|f\|_{L^p\p{\Omega^n}}^p\simeq_{M,p} \sum_{m\leq M} \overbrace{\int_{\Omega^n}\cdots \int_{\Omega^n}}^m \p{\sum_{1\leq i_1<\ldots<i_m\leq n} f_{i_1,\ldots,i_m} \p{x^{(1)}_{i_1},\ldots,x^{(m)}_{i_{m}}}^2}^\frac{p}{2}\d x^{(1)}\ldots \d x^{(m)}.\] \end{cor}
\begin{proof}This is just a combination of Theorem \ref{bourgsqfn} and Corollary \ref{multizinn}. \end{proof}
Again, for $p\geq 2$ this has been handled by multivariate extension of Rosenthal inequality, see \cite{GLZ}. For $1\leq p\leq 2$, by calling $\left|f_i\right|^p$ the new $f_i$ and $\frac{2}{p}$ the new $p$, we arrive at expressions of the form 
\beq \label{eq:moooi}\overbrace{\int_{\Omega^n}\cdots \int_{\Omega^n}}^m \p{\sum_{1\leq i_1<\ldots<i_m\leq n} f_{i_1,\ldots,i_m} \p{x^{(1)}_{i_1},\ldots,x^{(m)}_{i_{m}}}^p}^{\frac{1}{p}}\d x^{(1)}\ldots \d x^{(m)},\eeq
which the main object of our interest will be. For this, we can relax the condition on $p$ in \eqref{eq:moooi} to $1\leq p<\infty$ and the sum may be run over $i\in [1,n]^m$ instead of $1\leq i_1<\ldots<i_m\leq n$. 
\subsection{The case $m=2$} We are going to single out the case $m=2$, because we believe that it will significantly improve the legibility of the proof. 
\begin{thm}\label{4summand}Let $f_{i,j}\in L^1\left(\Omega^2\right)$ for $i,j=1,\ldots,n$ and $1\leq p<\infty$. Then
\begin{eqnarray}\label{eq:4sumequiv}\int_{\Omega^n}\int_{\Omega^n}\left(\sum_{i,j}\left|f_{i,j}\left(x_i,y_j\right)\right|^p\right)^\frac{1}{p} \d x\d y\simeq \inf_{\coprod_{i,j}f_{i,j}=a+b+c+d} &\|a\|_{L^1\left(\coprod_i \Omega \times \coprod_j \Omega\right)}\\ \nonumber +&\|b\|_{L^p\left(\coprod_i \Omega \times \coprod_j \Omega\right)}\\ \nonumber + &\|c\|_{L^1\left(\coprod_i \Omega,L^p\left(\coprod_j \Omega\right)\right)}\\ \nonumber+ &\|d\|_{L^1\left(\coprod_j \Omega,L^p\left(\coprod_i \Omega\right)\right)}\end{eqnarray}
with a constant not dependent on $p$. In more explicit terms, the inequality `$\gtrsim$' means that if \beq \label{eq:lesssim1}\int_{\Omega^n}\int_{\Omega^n}\left(\sum_{i,j}\left|f_{i,j}\left(x_i,y_j\right)\right|^p\right)^\frac{1}{p} \d x\d y\leq 1,\eeq then there is a decomposition \beq \label{eq:decompabcd}f_{i,j}=a_{i,j}+b_{i,j}+c_{i,j}+d_{i,j}\eeq such that 
\beq \label{eq:4suma}\sum_{i,j}\int_{\Omega}\int_{\Omega}\left|a_{i,j}\left(\xi ,\upsilon \right)\right| \d \xi \d \upsilon\lesssim 1,\eeq
\beq \label{eq:4sumb}\p{\sum_{i,j}\int_{\Omega}\int_{\Omega}\left|b_{i,j}\left(\xi ,\upsilon \right)\right|^p\d \xi \d \upsilon }^\frac{1}{p}\lesssim 1,\eeq
\beq \label{eq:4sumc}\sum_i \int_{\Omega}\left(\sum_j \int_{\Omega}\left|c_{i,j}\left(\xi ,\upsilon \right)\right|^p\d \upsilon \right)^\frac{1}{p} \d \xi \lesssim 1,\eeq
\beq \label{eq:4sumd}\sum_j \int_{\Omega}\left(\sum_i \int_{\Omega}\left|d_{i,j}\left(\xi ,\upsilon \right)\right|^p\d \xi \right)^\frac{1}{p} \d \upsilon \lesssim 1.\eeq 
Moreover, it can be chosen in such a way that for any $i,j$, supports of $a_{i,j},b_{i,j},c_{i,j},d_{i,j}$ are disjoint.\end{thm}
From this, we immediately derive
\begin{cor}\label{4summandmz}For $1\leq p\leq 2$ and $f\in U^p_2\p{\Omega^n}$ such that $f(x)=\sum_{i<j} f_{i,j}\p{x_i,x_j}$, 
\begin{align} \left\|f\right\|_{L^p\p{\Omega^n}}\simeq \inf_{f_{i,j}=a_{i,j}+b_{i,j}+c_{i,j}+d_{i,j}}& \p{\sum_{i<j}\int_{\Omega^2}\left|a_{i,j}\p{\xi,\upsilon}\right|^2\d\xi\d \upsilon}^\frac{1}{2}\\ 
+& \p{\sum_{i<j}\int_{\Omega^2}\left|b_{i,j}\p{\xi,\upsilon}\right|^p\d\xi\d \upsilon}^\frac{1}{p}\\
+& \p{ \sum_i\int_{\Omega}\p{\sum_{j>i}\int_{\Omega^2}\left|c_{i,j}\p{\xi,\upsilon}\right|^2\d \upsilon}^\frac{p}{2}\d\xi}^\frac{1}{p}\\
+& \p{ \sum_j\int_{\Omega}\p{\sum_{i<j}\int_{\Omega^2}\left|d_{i,j}\p{\xi,\upsilon}\right|^2\d\xi}^\frac{p}{2}\d \upsilon}^\frac{1}{p}.\end{align}
Moreover, the decomposition can be chosen such that $a_{i,j},b_{i,j},c_{i,j},d_{i,j}$ are mean zero in each variable. 
\end{cor}
\begin{proof}Just as indicated above, we use Corollary \ref{sqfndecoup} and then, in a routine convexification argment, Theorem \ref{4summand} for $\frac{2}{p}$ and $\left|f_{i,j}\right|^p$ to get the desired equivalence. The resulting decomposition is then improved by putting all summands to $0$ for $i\geq j$ and replacing $a_{i,j}$ etc. by $P_2a_{i,j}$ etc. (here, $P_2$ acts on functions on $\Omega^2$, namely $P_2=\p{\id-\E}\otimes\p{\id-\E}$), which is legal because
\beq f_{i,j}= P_2f_{i,j}= P_2a_{i,j}+P_2b_{i,j}+P_2c_{i,j}+P_2d_{i,j}.\eeq\end{proof}



\begin{proof}[Proof of Theorem \ref{4summand}.] Since all the norms involved, and consequently their interpolation sums, depend only on the modulus of a function, we may assume $f_{i,j}\geq 0$. By $\|\cdot\|_{\ell^p}\leq \|\cdot\|_{\ell^1}$ and $\|\cdot\|_{L^1}\leq \|\cdot\|_{L^p}$ we have
\[\int_{\Omega^n}\int_{\Omega^n}\left(\sum_{i,j}f_{i,j}\left(x_i,y_j\right)^p\right)^\frac{1}{p} \d x\d y\leq \int_{\Omega^n}\int_{\Omega^n}\sum_{i,j}\left|f_{i,j}\left(x_i,y_j\right)\right|\d x\d y \] \[=\sum_{i,j}\int_{\Omega}\int_{\Omega}\left|f_{i,j}\left(x_i,y_j\right)\right|\d x_i\d y_j= \left\|\coprod_{i,j}f_{i,j}\right\|_{L^1\left(\coprod_i\Omega\times \coprod_j\Omega\right)},\]

\[\int_{\Omega^n}\int_{\Omega^n}\left(\sum_{i,j}f_{i,j}\left(x_i,y_j\right)^p\right)^\frac{1}{p} \d x\d y\leq \sqrt{\int_{\Omega^n}\int_{\Omega^n}\sum_{i,j}f_{i,j}\left(x_i,y_j\right)^p\d x\d y }\] \[=\sqrt{\sum_{i,j}\int_{\Omega}\int_{\Omega}f_{i,j}\left(x_i,y_j\right)^p\d x_i\d y_j }= \left\|\coprod_{i,j}f_{i,j}\right\|_{L^p\left(\coprod_i\Omega\times \coprod_j\Omega\right)},\]

\[\int_{\Omega^n}\int_{\Omega^n}\left(\sum_{i,j}f_{i,j}\left(x_i,y_j\right)^p\right)^\frac{1}{p} \d x\d y\leq \int_{\Omega^n}\sum_i\sqrt{\int_{\Omega^n}\sum_{j}f_{i,j}\left(x_i,y_j\right)^p\d y} \d x\]
\[=\sum_i\int_\Omega\sqrt{\sum_j\int_\Omega f_{i,j}\left(x_{i},y_{j}\right)^p \d y_j}\d x_i= \left\|\coprod_{i,j}f_{i,j}\right\|_{L^1\left(\coprod_i\Omega,L^p\left(\coprod_j\Omega\right)\right)},\]

\[\int_{\Omega^n}\int_{\Omega^n}\left(\sum_{i,j}f_{i,j}\left(x_i,y_j\right)^p\right)^\frac{1}{p} \d x\d y\leq \int_{\Omega^n}\sum_j\p{\int_{\Omega^n}\sum_{i}f_{i,j}\left(x_i,y_j\right)^p\d x}^\frac12 \d y\]
\[=\sum_j\int_\Omega\p{\sum_i\int_\Omega f_{i,j}\left(x_{i},y_{j}\right)^p \d x_i}^\frac12\d y_j= \left\|\coprod_{i,j}f_{i,j}\right\|_{L^1\left(\coprod_j\Omega,L^p\left(\coprod_i\Omega\right)\right)}.\]
Thus, the $\lesssim$ inequality of \eqref{eq:4sumequiv} follows with constant $1$ and the only interesting part is the $\gtrsim$ inequality, i.e. the existence of a decomposition satisfying \eqref{eq:4suma}-\eqref{eq:4sumd}. \par
We perform the discretization procedure as described in the proof of Theorem \ref{JSch}, with the only changes being that as the dense set we choose functions that attain finitely many values, each (possibly treated with repetitions) of them on a set of product form, and our atoms also have to be products of atoms in $\Omega$. This allows us to assume that $\Omega$ is finite and equipped with the counting measure. Also, because both sides are lattice norms, we will be content with $a,b,c,d$ satisfying the desired ineqaulity and $a_{i,j}+b_{i,j}+c_{i,j}+d_{i,j}\geq f_{i,j}$ instead of $=f_{i,j}$. \par
Let $F_i$ be a function on $\Omega\times \Omega^n$ defined by the formula 
\beq F_i\left(\xi,y\right)=\p{\sum_j f_{i,j}\left(\xi,y_j\right)}^\frac{1}{p}.\eeq
We can write \eqref{eq:lesssim1} in the form 
\beq\nonumber\label{eq:bigfless1}\int_{\Omega^n} \int_{\Omega^n}\sqrt{\sum_i F_i\left(x_i,y\right)^p}\d x \d y\leq 1.\eeq 
By fixing $y$ and applying Theorem \ref{JSch} to the sequence of functions $F_i\left(\cdot,y\right)$, we get a decomposition 
\[F_i\left(x_i,y\right)= G_i\left(x_i,y\right)+H_i\left(x_i,y\right)\] 
such that 
\beq \label{eq:bigG}\sum_i\int_\Omega \left|G_i\left(\xi,y\right)\right|\d \xi\lesssim \int_{\Omega^n}\p{\sum_i F_i\left(x_i,y\right)^p}^\frac{1}{p}\d x,\eeq
\beq\label{eq:bigH}\p{\sum_i\int_\Omega H_i\left(\xi,y\right)^p\d \xi}^\frac{1}{p}\lesssim \int_{\Omega^n}\p{\sum_i F_i\left(x_i,y\right)^p}^\frac{1}{p}\d x.\eeq
Utilising the condition $\supp G_i\left(\cdot,y\right)\cap \supp H_i\left(\cdot,y\right)=\emptyset$ for all $y$ gives 
\[G_i\left(\xi,y\right)=w_i\left(\xi,y\right)F_i\left(\xi,y\right),\quad H_i\left(\xi,y\right)=\left(1-w_i\left(\xi,y\right)\right)F_i\left(\xi,y\right)\] for some sequence of functions $w_i:\Omega\times \Omega^n\to\{0,1\}$. Plugging into \eqref{eq:bigG} and integrating with respect to $y$ we get
\begin{eqnarray}\nonumber 1&\gtrsim& \int_{\Omega^n}\sum_i \int_{\Omega}\left|w_i\left(\xi ,y\right)F_{i}\left(\xi ,y\right)\right|\d \xi  \d y\\ &=& \label{eq:l1part}\sum_i \int_{\Omega}\int_{\Omega^n}\p{\sum_j \left(w_i\left(\xi ,y\right)f_{i,j}\left(\xi ,y_j\right)\right)^p}^\frac{1}{p}\d y\d \xi .\end{eqnarray} 
Let us fix $i$, $\xi $ and denote \beq \label{eq:wdefn}W_{i,j}\left(\xi ,\upsilon \right)=\mathbbm{1}_{\left\{\left(\E_j w_i\right)\left(\xi ,\cdot\right)\geq \frac12\right\}}\left(\upsilon \right).\eeq 
Applying Theorem \ref{premainlemma} in a setting where $|J|=1$, 
\beq \label{eq:applemma}\int_{\Omega^n}\p{\sum_j \left(w_i\left(\xi ,y\right)f_{i,j}\left(\xi ,y_j\right)\right)^p}^\frac{1}{p}\d y\gtrsim \int_{\Omega^n}\p{\sum_j \left(W_{i,j}\left(\xi ,y_j\right)f_{i,j}\left(\xi ,y_j\right)\right)^p}^\frac1p\d y.\eeq Theorem \ref{JSch} applied to functions $W_{i,j}f_{i,j}\left(x_i,\cdot\right)$ for $j=1,\ldots,n$ provides $\{0,1\}$-valued functions $u_{i,j}$ such that 
\beq\label{eq:aineq}\sum_j\int_\Omega \left|u_{i,j}W_{i,j}f_{i,j}\left(\xi ,\upsilon \right)\right|\d \upsilon \lesssim \int_{\Omega^n}\p{\sum_j \left(W_{i,j}f_{i,j}\left(\xi ,y_j\right)\right)^p}^\frac1p\d y,\eeq 
 \beq\p{\sum_j\int_\Omega \left(\left(1-u_{i,j}\right)W_{i,j}f_{i,j}\left(\xi ,\upsilon \right)\right)^p\d \upsilon }^\frac1p\lesssim \int_{\Omega^n}\p{\sum_j \left(W_{i,j}f_{i,j}\left(\xi ,y_j\right)\right)^p}^\frac1p\d y.\eeq
We may now put \[a_{i,j}=u_{i,j}W_{i,j}f_{i,j},\quad c_{i,j}= \left(1-u_{i,j}\right)W_{i,j}f_{i,j}.\] The desired inequality for $a_{i,j}$ is obtained by integrating \eqref{eq:aineq} and \eqref{eq:applemma} with respect to $\xi$, summing over $i$ and plugging into \eqref{eq:l1part}, and analogously for $c_{i,j}$.\par 
By integrating \eqref{eq:bigH} with respect to $y$ we get 
\begin{eqnarray} \nonumber 1&\gtrsim& \int_{\Omega^n}\p{\sum_i\int_\Omega \p{\left(1-w_{i}\right)\left(\xi ,y\right)F_i\left(\xi ,y\right)}^p\d \xi }^\frac{1}{p}\d y\\ \label{eq:lppart} &=& \int_{\Omega^n} \p{\sum_i \int_\Omega \sum_j \p{\p{1-w_i}\p{\xi ,y}f_{i,j}\p{\xi ,y_j}}^p\d \xi }^\frac1p\d y.\end{eqnarray} 
We may now incorporate $i$ and $x_i$ into one variable running through the space $\coprod_i \Omega$. Since $\Omega$ was equipped with the counting measure, $\coprod_i \Omega$ also is, up to a constant. This puts us in the setting of Theorem \ref{premainlemma} and allows to split the sequence of functions $y\mapsto \p{1-w_{i}}\p{\xi,y}f_{i,j}\p{\xi,y_j}$ into an $L^1\p{\bigsqcup_j \Omega,\ell^p\p{\coprod_i \Omega}}$ part and an $L^p\p{\bigsqcup_j \Omega,\ell^p\p{\coprod_i \Omega}}$ one. Bearing in mind that
\beq\mathbbm{1}_{\left\{\left(\E_j \p{1-w_i}\right)\left(\xi,\cdot\right)\geq \frac{1}{2}\right\}}\left(\upsilon\right)\geq 1-\mathbbm{1}_{\left\{\left(\E_j w_i\right)\left(\xi,\cdot\right)\geq \frac12\right\}}\left(\upsilon\right)=1-W_{i,j}\p{\xi,\upsilon},\eeq
we may apply Theorem \ref{premainlemma} to obtain functions $s_{i,j}:\Omega^2\to\{0,1\}$ such that
\beq\begin{aligned}
\int_{\Omega^n} \sqrt{\sum_i \int_\Omega \sum_j \p{\p{1-w_i}\p{\xi ,y}f_{i,j}\p{\xi ,y_j}}^p\d \xi }\d y &\gtrsim\\
\sum_j\int_\Omega \p{ \sum_i \int_\Omega \p{ s_{i,j}\p{1-W_{i,j}}f_{i,j}\p{\xi,\upsilon}}^p\d \xi}^\frac1p\d \upsilon& +\\
\p{ \sum_j\int_\Omega \sum_i \int_\Omega \p{ \p{1-s_{i,j}}\p{1-W_{i,j}}f_{i,j}\p{\xi,\upsilon}}^p\d \xi\d \upsilon}^\frac1p.
\end{aligned}\eeq
This allows us to take
\beq d_{i,j}= s_{i,j}\p{1-W_{i,j}},\quad b_{i,j}=\p{1-s_{i,j}}\p{1-W_{i,j}}.\eeq
The desired inequalities for $b_{i,j}$ and $d_{i,j}$ are directly verified. By definition of $a_{i,j},b_{i,j},c_{i,j},d_{i,j}$, we have $a_{i,j}+b_{i,j}+c_{i,j}+d_{i,j}=f_{i,j}$. They are also disjointly supported due to $W_{i,j},u_{i,j},s_{i,j}$ being $\{0,1\}$-valued, which ends the proof.\end{proof} \par

\subsection{The general case} \label{ssdecomp2m}We will now prove Theorem \ref{4summand} in full generality. For brevity, we will denote $\coprod_{i=1}^n\Omega$ by $\overline{\Omega}$, variables running through $\overline{\Omega}$ by $\overline{x},\overline{y},\ldots$ and write $\d \overline{x}= \d\p{ \coprod_i \mu}\p{\overline{x}}$. For example, if $\varphi_i\in L^1\p{\Omega}$, then \[\int_{\overline\Omega}\p{\coprod_i \varphi_i}\p{\overline{x}}\d\overline{x}= \sum_i\int_\Omega \varphi_i\p{x_i}\d x_i.\]

\begin{thm}\label{2msummand}
For $i=\p{i_1,\ldots,i_m}\in\{1,\ldots,n\}^m$, let $f_i\in L^1\p{\Omega^m}$. Suppose that
\beq \label{eq:main1stline}\overbrace{\int_{\Omega^{n}}\cdots\int_{\Omega^{n}}}^m \sqrtp{\sum_i f_i\p{x^{(1)}_{i_1},\ldots,x^{(m)}_{i_m}}^p}\d x^{(1)}\ldots \d x^{(m)}=1.\eeq 
Then, treating $\coprod_{i}f_i$ as a function on $\overline{\Omega}^m$, we have
\beq \label{eq:main2ndline}1\leq \inf_{\coprod_i f_i=\sum_{J\subset [1,m]}a^{(J)}}\sum_{J\subset [1,m]} \left\|a^{(J)}\right\|_{L^1\p{\overline{\Omega}^{J'},L^p\p{\overline{\Omega}^J}}}\leq C_m \eeq and $a^{(J)}= \coprod_i a^{(J)}_i$ can be chosen such that for any $i$, the supports of $a^{(J)}_i$ for different $J$ are disjoint. 
\end{thm}
\textit{Proof.} As previously, $\Omega$ is a finite set with counting measure and $f_i$ are nonnegative. In order to show the first inequality of \eqref{eq:main2ndline}, we just need to check that each of the norms $\left\|\cdot\right\|_{L^1\p{\overline{\Omega}^{J'},L^p\p{\overline{\Omega}^J}}}$ on functions on $\overline{\Omega}^m$ dominates the norm on the left hand side of \eqref{eq:main1stline}. Indeed, if $\ol{a}=\coprod_i a_i$, 

\begin{eqnarray} \lefteqn{ \int_{\Omega^{nm}} \sqrtp{\sum_i a_i\p{x^{(1)}_{i_1},\ldots,x^{(m)}_{i_m}}^p} \d x^{(1)}\ldots \d x^{(m)}}\\ 
&\leq& \label{eq:mspltrivjen}\int_{\p{\Omega^{n}}^{J'}} \sqrtp{ \int_{\p{\Omega^n}^J}\sum_{i_J,i_{J'}}a_{i_J,i_{J'}}\p{x^{(J)}_{i_J},x^{(J')}_{i_{J'}}}^p\d x^{(J)}}\d x^{(J')}\\
&\leq& \label{eq:mspltrivsum}\sum_{i_{J'}}\int_{\p{\Omega^{n}}^{J'}} \sqrtp{ \sum_{i_J}\int_{\p{\Omega^n}^J}a_{i_J,i_{J'}}\p{x^{(J)}_{i_J},x^{(J')}_{i_{J'}}}^p\d x^{(J)}}\d x^{(J')}\\
&=& \label{eq:mspltrivequal}\sum_{i_{J'}}\int_{\Omega^{J'}} \sqrtp{\sum_{i_J}\int_{\Omega^J} a_{i_J,i_{J'}}\p{x^{(J)}_{i_J},x^{(J')}_{i_{J'}}}^p\d x^{(J)}_{i_J}}\d x^{(J')}_{i_{J'}}\\ 
&=& \int_{\ol\Omega^{J'}}\sqrtp{ \int_{\ol\Omega^{J}}\ol{a}\p{\ol{x}^{(J)},\ol{x}^{(J')}}^p\d \ol{x}^{(J)}}\d \ol{x}^{(J')}\\
&=&\label{eq:trivdirl1l2}\left\|\ol{a}\right\|_{L^1\p{\overline{\Omega}^{J'},L^p\p{\overline{\Omega}^J}}}.
\end{eqnarray}
Here, we used abbreviations $i_J= \p{i_j}_{j\in J}$, $x^{(J)}=\p{x^{(j)}}_{j\in J}$, $x^{(J)}_{i_J}= \p{x^{(j)}_{i_j}}_{j\in J}$). The inequality \eqref{eq:mspltrivjen} is $\|\cdot\|_{L^1}\leq \|\cdot\|_{L^p}$ with respect to $x^{(J')}$, \eqref{eq:mspltrivsum} is $\sqrtp{\sum_{i_{J'}}c_{J'}}\leq \sum_{i_{J'}} \sqrtp{c_{i_{J'}}}$, and \eqref{eq:mspltrivequal} comes from the fact that for fixed $i_J$, the integrand depends on $x^{(J)}$ only through $x^{(J)}_{i_J}$. \par
We will now prove the second inequality of \eqref{eq:main2ndline} by induction with respect to $m$. Let us assume that the theorem is true for some $m\geq 1$. The discretization procedure is performed as previously. Let us take $f_{i,k}\in L^1\p{\Omega^{m+1}}$ for $i_1,\ldots,i_m,k\in [1,n]$, $i=\p{i_j}_{1\leq j\leq m}$. We define $F_i\in L^1\p{\Omega^m\times \Omega^n}$ by
\[F_i\p{x_1,\ldots,x_m,y}= \sqrtp{\sum_k f_{i,k}\p{x_1,\ldots,x_m,y_k}^p}\]
for $x_1,\ldots,x_m\in \Omega$ and $y\in\Omega^n$. For a fixed $y$, applying the induction hypothesis to the functions $F_i\p{\cdot,y}\in L^1\p{\Omega^m}$, yields a family of functions $w^{(J)}_i:\Omega^m\times \Omega^n\to \{0,1\}$ such that
\beq \label{eq:sumw1}\sum_J w^{(J)}_i=1\eeq
\beq \label{eq:indhypo} \sum_J \left\|\coprod_i\p{w^{(J)}_i F_i}\right\|_{L^1\p{\ol{\Omega}^{J'}, L^p\p{\ol{\Omega}^{J}}}}\leq C_m \int_{\Omega^{nm}}\sqrtp{\sum_i F_i\p{x^{(\leq m)}_{i_{\leq m}},y}^p}\d x^{(\leq m)}\eeq
(we just take $w^{(J)}_i= a^{(J)}_i/F_i$, because without loss of generality $a^{(J)}_i\p{x,y}=0$ for all $J$ if $F_i(x,y)=0$). We used another abbreviation: $x^{(\leq m)}_{i_{\leq m}}= \p{x^{(j)}_{i_j}}_{j\leq m}$. Let us define
\beq \label{eq:defbigW}W^{(J)}_{i,k}\p{x,y_k}= \mathbbm{1}_{\left\{  \E_k w^{(J)}_i\p{x,\cdot} \geq 2^{-m} \right\}}\p{x, y_k},\eeq where $\E_k$ is taken with respect to the second variable running through $\Omega^n$. \par
Denote $\ol{w}^{(J)}\p{\cdot,y}= \coprod_i w^{(J)}_i\p{\cdot,y}$, $\ol{W}^{(J)}_k\p{\cdot,y}= \coprod_i W^{(J)}_{i,k}\p{\cdot,y}$, $\ol{f}_k\p{\cdot,y}= \coprod_i f_{i,k}\p{\cdot,y}$. Applying \eqref{eq:indhypo} at each $y$, and then Theorem  \ref{premainlemma} at each $J$ and $\ol{x}^{(J')}$ (treating the integral over $\ol{\Omega}^J$ as summation over a finite set) we get
\begin{eqnarray*} \lefteqn{ C_m\int_{\Omega^n}\d y \int_{\Omega^{nm}} \d x^{(\leq m)} \sqrtp{\sum_{i,k} f_{i,k}\p{x^{(\leq m)}_{i_{\leq m}},y_k}^p} } \\
&=& C_m\int_{\Omega^n}\d y \int_{\Omega^{nm}} \d x^{(\leq m)} \sqrtp{\sum_{i} F_{i}\p{x^{(\leq m)}_{i_{\leq m}},y}^p} \\
&\geq& \int_{\Omega^n}\d y \sum_J \left\|\coprod_i\p{w^{(J)}_i F_i}\right\|_{L^1\p{\ol{\Omega}^{J'}, L^p\p{\ol{\Omega}^{J}}}}\\ 
&=& \sum_J \int_{\ol{\Omega}^{J'}}\d \ol{x}^{(J')} \int_{\Omega^n} \d y \sqrtp{ \int_{\ol{\Omega}^J} \sum_k \ol{w}^{(J)}\p{\ol{x},y}\ol{f}_k\p{\ol{x},y_k}^p\d \ol{x}^{(J)}} \\
&\gtrsim& 2^{-2m} \sum_J \int_{\ol{\Omega}^{J'}}\d \ol{x}^{(J')} \int_{\Omega^n} \d y \sqrtp{ \int_{\ol{\Omega}^J} \sum_k \ol{W}_k^{(J)}\ol{f}_k\p{\ol{x},y_k}^p\d \ol{x}^{(J)}}\\ 
&=& 2^{-2m} \sum_J \int_{\ol{\Omega}^{J'}}\d \ol{x}^{(J')} \int_{\Omega^n} \d y \sqrtp{ \sum_k \ol{\varphi}^{(J)}_k\p{\ol{x}^{(J')},y_k}^p },  \end{eqnarray*}
where 
\[ \ol{\varphi}_k^{(J)}\p{\ol{x}^{(J')},y_k}= \sqrtp{\int_{\ol{\Omega}^J} \ol{W}_k^{(J)}\ol{f}_k\p{\ol{x}^{(J)},\ol{x}^{(J')},y_k} \d \ol{x}^{(J)} }.\]
For each $J$ and $\ol{x}^{(J')}$, we can apply Theorem \ref{JSch} to functions $\ol{\varphi}_k^{(J)}\p{\ol{x}^{(J')},y_k}$. This produces $u^{(J)}_k:\ol{\Omega}^{J'}\times \Omega\to \{0,1\}$ such that 
\begin{eqnarray*} \int_{\Omega^n} \d y \sqrtp{ \sum_k \ol{\varphi}^{(J)}_k\p{\ol{x}^{(J')},y_k}^p }&\gtrsim& \sum_k \int_{\Omega} \ol{u}^{(J)}_k\ol{\varphi}_k^{(J)}\p{\ol{x}^{(J')},y_k}\d y_k\\ &+& \sqrtp{\sum_k \int_{\Omega} \p{1-\ol{u}^{(J)}_k}\ol{\varphi}_k^{(J)}\p{\ol{x}^{(J')},y_k}^p\d y_k}.\end{eqnarray*}
Plugging this into the previous inequality, we get
\begin{eqnarray*} \lefteqn{ \sum_J \int_{\ol{\Omega}^{J'}}\d \ol{x}^{(J')} \int_{\Omega^n} \d y \sqrtp{ \sum_k \ol{\varphi}^{(J)}_k\p{\ol{x}^{(J')},y_k}^p }}\\
&\gtrsim&  \sum_J \int_{\ol{\Omega}^{J'}}\d \ol{x}^{(J')} \sum_k \int_{\Omega} \ol{u}^{(J)}_k\ol{\varphi}_k^{(J)}\p{\ol{x}^{(J')},y_k}\d y_k\\ 
&\ & + \sum_J \int_{\ol{\Omega}^{J'}}\d \ol{x}^{(J')} \sqrtp{\sum_k \int_{\Omega} \p{1-\ol{u}^{(J)}_k}\ol{\varphi}_k^{(J)}\p{\ol{x}^{(J')},y_k}^p\d y_k}\\
&=&  \sum_J \int_{\ol{\Omega}^{J'}}\d \ol{x}^{(J')} \sum_k \int_{\Omega} \sqrtp{\int_{\ol{\Omega}^J} \ol{u}^{(J)}_k\ol{W}^{(J)}_k\ol{f}_k\p{\ol{x},y_k}^p\d \ol{x}^{(J)}} \d y_k\\ 
&\ & + \sum_J \int_{\ol{\Omega}^{J'}}\d \ol{x}^{(J')} \sqrtp{\sum_k \int_{\Omega} \int_{\ol{\Omega}^J} \p{1-\ol{u}^{(J)}_k}\ol{W}^{(J)}_k\ol{f}_k\p{\ol{x},y_k}^p\d \ol{x}^{(J)} \d y_k}\\ 
&=& \sum_{J\subset[1,m]} \int_{\ol{\Omega}^{J'}}\d \ol{x}^{(J')} \int_{\ol\Omega} \sqrtp{\int_{\ol{\Omega}^J} \coprod_k \overline{A}^{(J)}_{k}\p{\ol{x},\ol{y}}^p\d \ol{x}^{(J)}} \d \ol{y} \\
&\ & + \sum_{J\subset [1,m]} \int_{\ol{\Omega}^{J'}}\d \ol{x}^{(J')} \sqrtp{\int_{\ol\Omega} \int_{\ol{\Omega}^J} \coprod_k\overline{A}^{(J\cup\{m+1\})}_{k}\p{\ol{x},\ol{y}}^p\d \ol{x}^{(J)} \d \ol{y}}\\
&=& \sum_{J\subset [1,m]} \left\|\coprod_k \ol{A}^{(J)}_k\right\|_{L^1\p{\ol{\Omega}^{J'}\times \ol{\Omega},L^p\p{\ol{\Omega}^J}}}+ \sum_{J\subset [1,m]} \left\|\coprod_k \ol{A}^{(J\cup\{m+1\})}_k\right\|_{L^1\p{\ol{\Omega}^{J'},L^p\p{\ol{\Omega}^J\times \ol{\Omega}}}}\\
&=& \sum_{J\subset [1,m+1]}  \left\|\coprod_k \ol{A}^{(J)}_k\right\|_{L^1\p{\ol{\Omega}^{[1,m+1]\setminus J},L^p\p{\ol{\Omega}^J}}}\end{eqnarray*}
where the functions $A^{(J)}_{i,k}:\Omega^m\times\Omega\to \mathbb{R}$ for $J\subset[1,m+1]$ are defined by \[A^{(J)}_{i,k}(x,y)=\left\{\begin{matrix}u^{(J)}_{i,k}W^{(J)}_{i,k}f_{i,k}\p{x,y}& \text{ if }m+1\notin J\\ \p{1-u^{(J)}_{i,k}}W^{(J)}_{i,k}f_{i,k}\p{x,y}& \text{ if }m+1\in J.\end{matrix}\right.\]
Let us recall that by \eqref{eq:sumw1}, \[\sum_{J\subset[1,m]} \E_k w^{(J)}_i=1\] for all $i,k$ and consequently
\[\bigcup_{J\subset[1,m]} \left\{\E_k w^{(J)}_i\p{x,\cdot}\geq 2^{-m}\right\}=\Omega^n\] for any $i,k$ and $x\in\Omega^m$, which by \eqref{eq:defbigW} implies 
\[\sum_J W^{(J)}_{i,k}\geq 1.\]
Therefore 
\begin{eqnarray*} \left|\sum_{J\subset[1,m+1]}A^{(J)}_{i,k}\right| &=& \left|\sum_{J\subset[1,m]}W^{(J)}_{i,k}f_{i,k}\right|\\ 
&=& \left|\sum_{J\subset[1,m]}W^{(J)}_{i,k}\right| \left|f_{i,k}\right|\\ 
&\geq & \left|f_{i,k}\right|.\end{eqnarray*}
Ultimately, using the fact that $\sum_{J\subset[1,m+1]} L^1\p{\ol{\Omega}^{[1,m+1]\setminus J},L^p\p{\ol{\Omega}^J}}$ is a lattice, we get 
\begin{eqnarray*} \lefteqn{ C_m\int_{\Omega^n}\d y \int_{\Omega^{nm}} \d x^{(\leq m)} \sqrtp{\sum_{i,k} f_{i,k}\p{x^{(\leq m)}_{i_{\leq m}},y_k}^p} }\\
&\gtrsim & 2^{-2m} \left\| \sum_{J\subset[1,m+1]} \coprod_k \ol{A}_k^{(J)}\right\|_{\sum_{J\subset[1,m+1]} L^1\p{\ol{\Omega}^{[1,m+1]\setminus J},L^p\p{\ol{\Omega}^J}}} \\ 
&\geq & 2^{-2m} \left\| \coprod_{i,k} f_{i,k}\right\|_{\sum_{J\subset[1,m+1]} L^1\p{\ol{\Omega}^{[1,m+1]\setminus J},L^p\p{\ol{\Omega}^J}}} \end{eqnarray*}
as desired, with $C_{m+1}=C2^{2m}C_{m}$ for some numerical constant $C$. Once we have a decomposition verifying \eqref{eq:main2ndline}, for each $\ol{x}\in\ol{\Omega}^m$ we select $J_{\ol{x}}$ such that 
\[\left|a^{(J_{\ol{x}})}\p{\ol{x}}\right|\geq 2^{-m} \left|\ol{f}\p{\ol{x}}\right|.\]
Now, the functions 
\[\alpha^{(J)}\p{\ol{x}}= \left\{\begin{matrix} \ol{f}\p{\ol{x}}&\text{if } J=J_{\ol{x}}\\ 0 & \text{otherwise}\end{matrix}\right.\] are dominated by $a^{(J)}$ up to the constant $2^m$ and their sum over $J$ is $\ol{f}$. \qed\par
By an identical reasoning as previously, we get
\begin{cor}\label{2msummandmz}
Let $1\leq p\leq 2$ and $f\in U^p_m\p{\Omega^n}$ have a representation 
\beq f(x)=\sum_{1\leq i_1<\ldots<i_m\leq n} f_{i_1,\ldots,f_{i_m}}\p{x_{i_1},\ldots,x_{i_m}}.\eeq
Then
\beq \left\|f\right\|_{L^p}\simeq_m\inf_{f=\sum_{J\subset [1,m]}a^{(J)}}\sum_{J\subset [1,m]} \left\|\coprod_i a^{(J)}_i\right\|_{L^p\p{\overline{\Omega}^{J'},L^2\p{\overline{\Omega}^J}}},\eeq 
where the infimum is taken over decompositions of $f$ into summands $a^{(J)}\in U^p_m$ and 
\beq a^{(J)}(x)=\sum_{1\leq i_1<\ldots<i_m\leq n} a^{(J)}_i\p{x_{i_1},\ldots,x_{i_m}}.\eeq
\end{cor}

\section{Interpolation of $U^1_m\p{\Omega^\infty,L^p(\mathbb{R})}$}\label{secinterp}
We can extend the definition of spaces $\V_m^p$ to the vector-valued setting. Let $B$ be a Banach space. Then we define $\V_m^p\p{\Omega^n,B}$ as the closure of $\V_m^p\p{\Omega^n}\otimes B$ in the Bochner space $L^p\p{\Omega^n,B}$. In particular, $\V_1^p$ consists of functions of the form $f(x)=\sum_i f\p{x_i}$, where $f_i\in L^p\p{\Omega,B}$ and $\int_\Omega f_i\p{x_i}\d x_i=0$. By combining Corollary \ref{sqfndecoup} with Lemma \ref{rbdd}, we get
\begin{cor}\label{decoupvectorval}Let $f\in \V_{\leq M}^1\p{\Omega^n,B}$, where $B$ is a Hilbert space. Then 
\beq \|f\|_{L^1\p{\Omega^n,B}}\simeq_M \sum_{0\leq m\leq M}\int_{\Omega^{nm}}\sqrt{\sum_{i_1<\ldots<i_m}\left\|f_i\p{ \p{x^{(j)}_{i_j}}_{j=1,\ldots,m}}\right\|_B^2}\d x^{(\leq m)}.\eeq
\end{cor}

\begin{deff}Let $X\subset L^1(S)$ and $X\p{B}$ be generated by $X\otimes B$ in $L^1\p{S,B}$. The subspace $X$ is said to have Bourgain-Pisier-Xu (BPX) property if $\p{X\p{\ell^1},X\p{\ell^p}}$ is $K$-closed in $\p{L^1\p{S,\ell^1},L^1\p{S,\ell^p}}$ for some $1<p<\infty$. 
\end{deff}
The following is a result of Xu \cite[Proposition 11]{xul1h1}, based on pieces of reasoning by Bourgain \cite{bourgcotype} and Pisier \cite{pisiersimple}. 
\begin{thm}[Bourgain, Pisier, Xu]If $X\subset L^1$ has BPX property, then $L^1/X$ is of cotype 2 and every operator $L^1/X\to\ell^2$ is 1-summing. 
\end{thm}

\subsection{The case $m=2$} As previously, we single out $m=2$. 
\begin{thm}\label{kclv1l1}The couple \[\p{\V^1_1\p{\Omega^n,L^1(\mathbb{R})},\V^1_1\p{\Omega^n,L^2(\mathbb{R})}}\] is $K$-closed in \[\p{L^1\p{\Omega^n,L^1(\mathbb{R})},L^1\p{\Omega^n,L^2(\mathbb{R})}},\] with a constant independent of $n$. \end{thm}
\textit{Proof.} For $f\in L^0\p{\Omega^n\times \mathbb{R}}$ and $\alpha>0$, let $\p{f\circ \alpha}(x,s)= f(x,\alpha s)$. To shorten the notation, we will denote the underlying couples by $\p{\V^1_1\p{L^1},\V^1_1\p{L^2}}$ and $\p{L^1\p{L^1},L^1\p{L^2}}$. Since \[\n{f\circ\alpha}_{L^1\p{L^1}}=\alpha^{-1}\n{f}_{L^1\p{L^1}},\quad \n{f\circ\alpha}_{L^1\p{L^2}}=\alpha^{-\frac12}\n{f}_{L^1\p{L^2}},\] we have \begin{eqnarray*}K\p{f\circ t^{-2},t;L^1\p{L^1},L^1\p{L^2}}&=&\inf_{f\circ t^{-2}=\p{g\circ t^{-2}}+ \p{h\circ t^{-2}}}\n{g\circ t^{-2}}_{L^1\p{L^1}}+ t\n{h\circ t^{-2}}_{L^1\p{L^2}}\\ &=&\inf_{f=g+h}t^2\n{g}_{L^1\p{L^1}}+ t^2\n{h}_{L^1\p{L^2}}\\ &=& t^2 K\p{f,1;L^1\p{L^1},L^1\p{L^2}}.\end{eqnarray*}
Analogously \[K\p{f\circ t^{-2},t;\V^1_1\p{L^1},\V^1_1\p{L^2}}= t^2 K\p{f,1;\V^1_1\p{L^1},\V^1_1\p{L^2}}.\] Therefore we only have to prove \[K\p{f,1;\V^1_1\p{L^1},\V^1_1\p{L^2}}\lesssim K\p{f,1;L^1\p{L^1},L^1\p{L^2}}\] for $f\in \V^1_1\p{L^1}+\V^1_1\p{L^2}$. \par
For any such $f$, the decomposition $f(x)=\sum_i f_i\p{x_i}$ is unique, because $f_i=\E_i f$. Let $f_{i,j}:\Omega\times [0,1]\to\mathbb{R}$, $f_j:\Omega^n\times [0,1]\to\mathbb{R}$ be defined by \[f_{i,j}\p{x_i,s_j}= f_i\p{x_i,s_j+j-1}, \quad f_j\p{x,s_j}=f\p{x,s_j+j-1}.\] Then $f_j(x)=\sum_i f_{i,j}\p{x_i}$ and we can treat $f_i$ as $\coprod_j f_{i,j}$ and $f$ as $\coprod_j f_j$, by identification of $(0,\infty)$ with $\coprod_j [0,1]$. Using the trivial part of Theorem \ref{JSch} at each $x$ separately, we get
\begin{eqnarray*}
K\p{f,1;L^1\p{L^1},L^1\p{L^2}}&=&\inf_{f=g+h}\int_{\Omega^n}\n{g(x)}_{L^1}+\n{h(x)}_{L^2}\d x\\ 
&=&\int_{\Omega^n}\inf_{f(x)=g(x)+h(x)}\n{g(x)}_{L^1}+\n{h(x)}_{L^2}\d x\\
&=&\int_{\Omega^n}\n{\coprod_j f_j(x,\cdot)}_{\p{L^1+L^2}(0,\infty)}\d x\\ 
&\gtrsim&\int_{\Omega^n} \int_{[0,1]^\infty}\sqrt{\sum_j f_j\p{x,s_j}^2}\d s\d x\\ 
&=&\int_{[0,1]^\infty} \int_{\Omega^n} \sqrt{\sum_j \p{\sum_i f_{i,j}\p{x_i,s_j}}^2}\d x\d s\\ 
&\simeq& \int_{[0,1]^\infty} \int_{\Omega^n} \sqrt{\sum_{i,j}f_{i,j}\p{x_i,s_j}^2}\d x\d s,
\end{eqnarray*}
where the last equivalence is Lemma \ref{marzyglp} applied at each $s$ to the sequence of $\ell^2$-valued independent mean $0$ functions $ \p{f_{i,j}\p{\cdot,s_j}}_j$. Let \[f_{i,j}=a_{i,j}+b_{i,j}+c_{i,j}+d_{i,j}\] be the decomposition given by Theorem \ref{4summand}. We can ensure that $a_{i,j}, b_{i,j},c_{i,j},d_{i,j}$ are of mean $0$ in the first variable, because $f_{i,j}$ are and subtracting the underlying conditional expectation preserves \eqref{eq:4suma}-\eqref{eq:4sumd}. Let \[a_i=\coprod_j a_{i,j},\quad b_{i}=\coprod_j b_{i,j},\quad c_{i}= \coprod_j c_{i,j},\quad d_i=\coprod_j d_{i,j},\] 
We have
\begin{eqnarray*}
\n{\sum_i a_i}_{L^1\p{L^1}}&=& \int_{\Omega^n} \int_{(0,\infty)}\left|\sum_i a_i\p{x_i,s}\right|\d s\d x\\
&\leq &\int_{(0,\infty)}\sum_i\int_\Omega \left|a_i\p{x_i,s}\right|\d x_i\d s\\ 
&=& \n{\coprod_i a_i}_{L^1\p{\coprod_i\Omega\times (0,\infty)}},
\end{eqnarray*}
\begin{eqnarray*} \n{\sum_i d_i}_{L^1\p{L^1}}&=& \int_{\Omega^n}\int_{(0,\infty)} \left|\sum_i d_i\p{x_i,s}\right|\d s\d x\\
&=&\int_{(0,\infty)} \int_{\Omega^n}\left|\sum_i d_i\p{x_i,s}\right|\d x\d s\\
&\simeq&\int_{(0,\infty)} \int_{\Omega^n}\sqrt{\sum_i d_i\p{x_i,s}^2\d x}\d s\\
&\leq &\int_{(0,\infty)} \sqrt{\sum_i\int_{\Omega} d_i\p{x_i,s}^2\d x_i}\d s\\
&=& \n{\coprod_i d_i}_{L^1\p{(0,\infty),L^2\p{\coprod_i\Omega}}},
\end{eqnarray*}
\begin{eqnarray*}\n{\sum_i c_i}_{L^1\p{L^2}}&=& \int_{\Omega^n}\n{\sum_i c_i\p{x_i}}_{L^2}\d x\\
&\simeq& \int_{\Omega^n}\n{\sqrt{\sum_i c_i\p{x_i}^2}}_{L^2}\d x\\
&=&\int_{\Omega^n}\sqrt{\int_{(0,\infty)}\sum_i c_i\p{x_i,s}^2\d s}\d x\\
&\leq & \int_{\Omega^n}\sum_i\sqrt{\int_{(0,\infty)} c_i\p{x_i,s}^2\d s}\d x\\
&=&\sum_i \int_{\Omega}\sqrt{\int_{(0,\infty)} c_i\p{x_i,s}^2\d s}\d x_i\\
&=&\n{\coprod_i c_i}_{L^1\p{\coprod_i\Omega,L^2(0,\infty)}},
\end{eqnarray*}
\begin{eqnarray*}\n{\sum_i b_i}_{L^1\p{L^2}}&=&\int_{\Omega^n}\n{\sum_i b_i\p{x_i}}_{L^2}\d x\\
&\simeq& \int_{\Omega^n}\n{\sqrt{\sum_i b_i\p{x_i}^2}}_{L^2}\d x\\
&\leq & \sqrt{\int_{\Omega^n}\n{\sqrt{\sum_i b_i\p{x_i}^2}}_{L^2}^2\d x}\\
&=& \sqrt{\int_{\Omega^n}\int_{(0,\infty)}\sum_i b_i\p{x_i,s}^2\d s\d x}\\
&=& \n{\coprod_i b_i}_{L^2\p{\coprod_i\Omega\times (0,\infty)}}.
\end{eqnarray*}
Ultimately, we put \[g_i= a_{i}+ d_{i},\quad h_{i}= c_{i}+b_{i}.\] By definition of $a_i,b_i,c_i,d_i$, we have $f(x)=\sum_i g_i\p{x_i}+\sum_i h_i\p{x_i}$. Combining the above inequalities,
\begin{eqnarray*}\lefteqn{ K\p{f,1;\V^1_1\p{L^1},\V^1_1\p{L^2}}}\\ 
&\leq& \n{\sum_i g_i}_{L^1\p{L^1}}+ \n{\sum_i h_i}_{L^1\p{L^2}}\\
&\leq & \n{\sum_i a_i}_{L^1\p{L^1}}+ \n{\sum_i d_i}_{L^1\p{L^1}}+ \n{\sum_i c_i}_{L^1\p{L^2}}+ \n{\sum_i b_i}_{L^1\p{L^2}}\\
&\lesssim& \n{\coprod_i a_i}_{L^1\p{\coprod_i\Omega\times (0,\infty)}}+ \n{\coprod_i d_i}_{L^1\p{(0,\infty),L^2\p{\coprod_i\Omega}}}+\\
&\ & \n{\coprod_i c_i}_{L^1\p{\coprod_i\Omega,L^2(0,\infty)}}+ \n{\coprod_i b_i}_{L^2\p{\coprod_i\Omega\times (0,\infty)}}\\
&\lesssim& \int_{[0,1]^\infty} \int_{\Omega^n} \sqrt{\sum_{i,j}f_{i,j}\p{x_i,s_j}^2}\d x\d s\\
&\lesssim& K\p{f,1;L^1\p{L^1},L^1\p{L^2}}
\end{eqnarray*}
as desired.\qed\par
\subsection{The general case} We are now prepared for the proof of the following, which by \cite{xul1h1}, is going to imply that $L^1\p{\Omega^\infty}/\V^1_{\leq M}\p{\Omega^\infty}$ is of cotype 2. In this subsection, $\lesssim,\gtrsim,\simeq$ are understood to depend on $M$ or $m$. 
\begin{thm}\label{kclvml1}The couple \[\p{\V^1_{\leq M}\p{\Omega^n,L^1(\mathbb{R})},\V^1_{\leq M}\p{\Omega^n,L^2(\mathbb{R})}}\] is $K$-closed in \[\p{L^1\p{\Omega^n,L^1(\mathbb{R})},L^1\p{\Omega^n,L^2(\mathbb{R})}},\] with a constant independent of $n$. \end{thm}
\textit{Proof.} Let us take $f\in \V^1_{\leq M}\p{\Omega^n,\p{L^1\p{\R}}}+\V^1_{\leq M}\p{\Omega^n,L^2\p{\R}}$ with a decomposition 
\[f(x)=\sum_{0\leq m\leq M}\sum_{i_1<\ldots<i_m}f_{i_1,\ldots,i_m}\p{x_{i_1},\ldots,x_{i_m}},\]
where 
\beq  \label{eq:fiinvmOm} f_{i_1,\ldots,i_m}\in \V^1_m\p{\Omega^m,L^1\p{\R}}+ \V^1_m\p{\Omega^m,L^2\p{\R}}.\eeq
By density argument, we may assume that $f_{i_1,\ldots,i_m}(x)$ vanishes outside of $[-N,N]$, which allows $f_{i_1,\ldots,i_m}$ to formally be treated as an element of $\V^1_m\p{\Omega^n,L^1\p{\R}}$. Just as in the proof of Theorem \ref{kclv1l1}, we reduce the problem to 
\beq K\p{f,1;\V^1_{\leq M}\p{L^1},\V^1_{\leq M}\p{L^2}}\lesssim K\p{f,1;L^1\p{L^1},L^1\p{L^2}}.\eeq 
Again, we denote the restriction of $f$ to $[k,k+1]$ in the variable running through $\R$ by $f_k$ and an analogous restriction of $f_i$ by $f_{i,k}$. As previously, we calculate the right hand side in preparation to use Theorem \ref{2msummand}.
\begin{eqnarray*}
\lefteqn{K\p{f,1;L^1\p{\Omega^n,L^1\p{\R}},L^1\p{\Omega^n,L^2\p{\R}}}}\\
&=&\inf_{f=g+h}\int_{\Omega^n}\n{g(x)}_{L^1\p{\R}}+\n{h(x)}_{L^2\p{\R}}\d x\\ 
&=&\int_{\Omega^n}\inf_{f(x)=g(x)+h(x)}\n{g(x)}_{L^1\p{\R}}+\n{h(x)}_{L^2\p{\R}}\d x\\
&=&\int_{\Omega^n}\n{\coprod_{k\in \Z} f_k(x)}_{\p{L^1+L^2}\p{\R}}\d x\\ 
&\geq& \int_{\Omega^n} \int_{[0,1]^\infty}\sqrt{\sum_k f_k\p{x,s_k}^2}\d s\d x\\ 
&=&\int_{[0,1]^\infty} \int_{\Omega^n} \left\|\p{f_k\p{x,s_k}}_{k\in\Z}\right\|_{\ell^2(\Z)}\d x\d s\\ 
&\simeq& \int_{[0,1]^\infty} \sum_{m\leq M} \int_{\Omega^{nm}}\sqrt{\sum_{i_1<\ldots<i_m}\left\|\p{f_{i,k}\p{x^{(\leq m)}_{i_{\leq m}},s_k}}_{k\in\Z} \right\|^2_{\ell^2(\Z)}}\d x^{(\leq m)}\d s\\ 
&=&\sum_{m\leq M}\int_{[0,1]^\infty} \int_{\Omega^{nm}} \sqrt{\sum_{i_1<\ldots<i_m}\sum_k f_{i,k}\p{x^{(\leq m)}_{i_{\leq m}},s_k}^2}\d x^{(\leq m)}\d s.
\end{eqnarray*}
Here, the `$\simeq$' inequality is an application of Corollary \ref{decoupvectorval} at every $\p{s_k}_{k\in Z}$. Since the $K$-functional 
\beq K\p{\cdot,1;\V^1_1\p{\Omega^n,L^1},\V^1_1\p{\Omega^n,L^2}}\eeq is a norm, we can without loss of generality fix $m$, assume that
\[f\p{x_1,\ldots,x_n}=\sum_{i_1<\ldots<i_m}f_{i_1,\ldots,i_m}\p{x_{i_1},\ldots,x_{i_m}}\] and aim at proving 
\begin{eqnarray}\label{eq:willbeenoughforkcl} K\p{f,1;\V^1_1\p{\Omega^n,L^1},\V^1_1\p{\Omega^n,L^2}} \lesssim\\ 
\nonumber\int_{[0,1]^\infty} \int_{\Omega^{nm}} \sqrt{\sum_{i_1<\ldots<i_m}\sum_k f_{i,k}\p{x^{(\leq m)}_{i_{\leq m}},s_k}^2}\d x^{(\leq m)}\d s.\end{eqnarray}
If we treat $s$ as the $m+1$-th set of variables, Theorem \ref{2msummand} applied to the sequence $\p{f_{i,k}}_{(i,k)\in [1,n]^m\times \Z}$ of functions in $L^1\p{\Omega^m\times [0,1]}$ (we take $f_{i,k}=0$ unless $i_1<\ldots<i_m$) gives a decomposition
\beq\label{eq:fintoab} f_{i,k}=\sum_{J\subset [1,m]}a_{i,k}^{(J)}+ \sum_{J\subset [1,m]}b_{i,k}^{(J)}\eeq
such that 
\begin{eqnarray} \label{eq:aplusbleqf}\sum_{J\subset [1,m]}\left\|\coprod_{i,k}a_{i,k}^{(J)}\right\|_{L^1\p{\ol{\Omega}^{[1,m]\setminus J}\times \R,L^2\p{\ol{\Omega}^J}}}+ \sum_{J\subset [1,m]}\left\|\coprod_{i,k}b_{i,k}^{(J)}\right\|_{L^1\p{\ol{\Omega}^{[1,m]\setminus J},L^2\p{\ol{\Omega}^J\times \R}}} \\
\nonumber \lesssim \int_{[0,1]^\infty} \int_{\Omega^{nm}} \sqrt{\sum_{i_1<\ldots<i_m}\sum_k f_{i,k}\p{x^{(\leq m)}_{i_{\leq m}},s_k}^2}\d x^{(\leq m)}\d s,\end{eqnarray}
where $a_{i,k}^{(J)}$ correspond to subsets of $[1,m+1]$ not containing $m+1$ and $b_{i,k}^{(J)}$ correspond to subsets of $[1,m+1]$ containing $m+1$. Just as remarked after the formulation of Theorem \ref{4summand}, we can modify $a_{i,k}^{(J)}$, $b_{i,k}^{(J)}$ by setting 
\beq \label{eq:lotsof0s} a_{i,k}^{(J)}=b_{i,k}^{(J)}=0\text{ unless }i_1<\ldots<i_m\eeq
and replacing $a_{i,k}^{(J)}$ by $\p{P_m\otimes\id}a_{i,k}^{(J)}$ and $b_{i,k}^{(J)}$ by $\p{P_m\otimes\id}b_{i,k}^{(J)}$, where $P_m$ acts on the first $m$ variables and $\id$ acts on the last. This operation is legal, because $P_m=\p{\id-\E}^{\otimes m}$ is a finite combination of conditional expectations. Moreover, by \eqref{eq:fiinvmOm} and \eqref{eq:fintoab} it produces a decomposition of $f_{i,k}$ into summands that are in $\V^1_m$ with respect to the first $m$ variables. Also, due to boundedness of $P_m$ in all mixed $L^1\p{L^2}$ norms, the left hand side of \eqref{eq:aplusbleqf} gets smaller up to a constant. Therefore, we can without loss of generality assume \eqref{eq:lotsof0s} and 
\beq \label{eq:abafterpm}a_{i,k}^{(J)}\p{\cdot,s_k},b_{i,k}^{(J)}\p{\cdot,s_k}\in \V^1_m\p{\Omega^m}\text{ for all }i,k,J,s_k,\eeq
making them suitable for forming 
\[ g_k\p{x_1,\ldots,x_n,s_k}= \sum_J \sum_{i_1<\ldots<i_m} a_{i,k}^{(J)}\p{x_{i_1},\ldots,x_{i_m},s_k},\]
\[ h_k\p{x_1,\ldots,x_n,s_k}= \sum_J \sum_{i_1<\ldots<i_m} b_{i,k}^{(J)}\p{x_{i_1},\ldots,x_{i_m},s_k},\]
both of which are in $\V^1_m$ with respect to the first $n$ variables thanks to \eqref{eq:abafterpm} and satisfy 
\beq \label{eq:finalfgh}\coprod_k g_k+\coprod_k h_k=\coprod_k f_k=f\eeq
because of \eqref{eq:fintoab}. Now, making use of Corollary \ref{sqfndecoup} in \eqref{eq:usedecoupl1val} and \eqref{eq:trivdirl1l2} in \eqref{eq:usetrivforg},
\begin{eqnarray} \lefteqn{\nonumber\left\|\coprod_k g_k\right\|_{\V^1_m\p{\Omega^n,L^1\p{\R}}}= }\\
\nonumber&=& \int_{\Omega^n}\left\|\coprod_k\sum_J\sum_i a_{i,k}^{(J)}\p{x_{i_1},\ldots,x_{i_m}}\right\|_{L^1\p{\R}}\d x\\
\nonumber&\leq& \sum_J \int_{\Omega^n}\left\|\coprod_k\sum_i a_{i,k}^{(J)}\p{x_{i_1},\ldots,x_{i_m}}\right\|_{L^1\p{\R}}\d x\\
\nonumber&=&\sum_J \int_{\R}\d \ol{s} \int_{\Omega^n}\d x \left|\sum_i\coprod_k a_{i,k}^{(J)}\p{x_{i_1},\ldots,x_{i_m},\ol{s}}\right|\\
&\simeq&\label{eq:usedecoupl1val} \sum_J \int_{\R}\d \ol{s} \int_{\Omega^{nm}}\d x^{(\leq m)} \sqrt{\sum_i \p{\coprod_k a_{i,k}^{(J)}\p{x_{i_{1}}^{(1)},\ldots,x_{i_m}^{(m)},\ol{s}}}^2}\\
&\leq & \label{eq:usetrivforg} \sum_J \int_{\R}\d\ol{s}\int_{\ol{\Omega}^{J'}} \sqrt{\int_{\ol{\Omega}^J}\coprod_{i,k} a_{i,k}^{(J)}\p{\ol{x}^{(1)},\ldots\ol{x}^{(m)},\ol{s}}^2\d \ol{x}^{(J)}}\d \ol{x}^{(J')}\\
\nonumber&=& \sum_{J\subset [1,m]}\left\|\coprod_{i,k}a_{i,k}^{(J)}\right\|_{L^1\p{\ol{\Omega}^{J'}\times \R,L^2\p{\ol{\Omega}^J}}}.
\end{eqnarray}
Similarly, using Corollary \ref{decoupvectorval} in \eqref{eq:usedecoupl2val} and \eqref{eq:trivdirl1l2} in \eqref{eq:usetrivforh},
\begin{eqnarray} \lefteqn{\nonumber\left\|\coprod_k h_k\right\|_{\V^1_m\p{\Omega^n,L^2\p{\R}}}= }\\
\nonumber&=& \int_{\Omega^n}\left\|\coprod_k\sum_J\sum_i b_{i,k}^{(J)}\p{x_{i_1},\ldots,x_{i_m}}\right\|_{L^2\p{\R}}\d x\\
\nonumber&\leq& \sum_J \int_{\Omega^n}\left\|\coprod_k\sum_i b_{i,k}^{(J)}\p{x_{i_1},\ldots,x_{i_m}}\right\|_{L^2\p{\R}}\d x\\
&\simeq&\label{eq:usedecoupl2val} \sum_J \int_{\Omega^{nm}}\d x^{(\leq m)} \sqrt{\sum_i \left\|\coprod_k b_{i,k}^{(J)}\p{x_{i_{1}}^{(1)},\ldots,x_{i_m}^{(m)}}\right\|_{L^2\p{\R}}^2}\\
&\leq & \label{eq:usetrivforh} \sum_J \int_{\ol{\Omega}^{J'}} \sqrt{\int_{\ol{\Omega}^J}\left\|\coprod_{i,k} b_{i,k}^{(J)}\p{\ol{x}^{(1)},\ldots,\ol{x}^{(m)}}\right\|_{L^2(\R)}^2\d \ol{x}^{(J)}}\d \ol{x}^{(J')}\\
&=&\nonumber \sum_J \int_{\ol{\Omega}^{J'}} \sqrt{\int_{\ol{\Omega}^J}\int_{\R}\coprod_{i,k} b_{i,k}^{(J)}\p{\ol{x}^{(1)},\ldots,\ol{x}^{(m)},\ol{s}}^2\d\ol{s}\d \ol{x}^{(J)}}\d \ol{x}^{(J')}\\
\nonumber&=& \sum_{J\subset [1,m]}\left\|\coprod_{i,k}b_{i,k}^{(J)}\right\|_{L^1\p{\ol{\Omega}^{J'},L^2\p{\ol{\Omega}^J\times \R}}}.
\end{eqnarray}
Summing up the last two inequalities and connecting them with \eqref{eq:aplusbleqf}, we see that \eqref{eq:finalfgh} defines a decomposition of $f$ verifying \eqref{eq:willbeenoughforkcl}, which ends the proof. \qed
\section{Interpolation of $\V^p_m\p{\Omega^\infty,L^p(\R)}$}

For $m\geq 2$, the projection $P_m$ is not a Calder\'on--Zygmund operator because its norm on $L^p$ behaves as $\p{\frac{p}{\log p}}^2$ for $p\to \infty$ and we know of no way to represent $P_1$ as such an operator. Nonetheless, we are able to show that Bourgain's result \cite{bourginterp} about $K$-closedness of an image of a C-Z projection in $\p{L^1,L^2}$ holds for $P_m$ as well. Here, we present only the case of $m=1$, while for the general $m$ the proof is, as previously, analogous but more technical. 
\begin{thm}\label{bourgu1m}The couple 
\beq \p{\V^1_1\p{\Omega^\infty,L^1\p{\R}},\V^2_1\p{\Omega^\infty,L^2\p{\R}}}\eeq
is $K$-closed in 
\beq \p{L^1\p{\Omega^\infty,L^1\p{\R}},L^2\p{\Omega^\infty,L^2\p{\R}}}.\eeq
\end{thm}
One can easily recover $K$-closedness of respective scalar-valued spaces by restricting to the subspace consisting of functions with values in $\mathrm{span}\left\{\mathbbm{1}_{[0,1]}\right\}$. It is of independent interest that the proof below will show that a sum of only 3 of 4 interpolation summands appearing in Theorem \ref{4summand} for $p=2$ also has a natural interpretation. 
\begin{proof}
Let $f\in \V^1_1\p{\Omega^\infty,L^1\p{\R}}+\V^2_1\p{\Omega^\infty,L^2\p{\R}}$. Then $f(x)=\sum_i f_i\p{x_i}$, where $f_i:\Omega\times \R\to \R$. Denoting by $f_{i,j}$ the restriction of $f_i$ to $\Omega\times [j,j+1]$, we can put $f_i=\bigsqcup_{j}f_{i,j}$. Our goal is to prove
\beq K\p{f,t;\V^1_1\p{\Omega^\infty,L^1\p{\R}},\V^2_1\p{\Omega^\infty,L^2\p{\R}}}\lesssim K\p{f,t;L^1\p{\Omega^\infty,L^1\p{\R}},L^2\p{\Omega^\infty,L^2\p{\R}}}.\eeq
Just like in the proof of the theorem about inteprolation of $\V^1_m\p{\ell^p}$, by means of scaling in $\R$ we can without loss of generality assume that $t=1$. Then the right hand side is the $L^1+L^2$ norm on the space $\bigsqcup_j \p{\Omega^\infty\times [0,1]}$, so by the trivial part of Johnson-Schechtman inequality
\beq K\p{f,1;L^1\p{\Omega^\infty,L^1\p{\R}},L^2\p{\Omega^\infty,L^2\p{\R}}}\geq \int_{\p{\Omega^\N\times[0,1]}^\Z}\d s\d x \sqrt{\sum_{j\in \Z} \left|\sum_{i\in\N} f_{i,j}\p{x_i^{(j)},s_j}\right|^2}.\eeq
Let us fix $s$ for a moment. The right hand side equals $\int \d x \left\| \sum_{i\in \N} \p{f_{i,j}\p{x^{(j)}_i,s_j}}_{j\in\Z}\right\|_{\ell^2\p{\Z}}$. Therefore, by Marcinkiewicz-Zygmund inequality, the sum over $i$ is unconditional, which allows us to write
\beq \int_{\p{\Omega^\N}^\Z}\d x \sqrt{\sum_{j\in \Z} \left|\sum_{i\in\N} f_{i,j}\p{x_i^{(j)},s_j}\right|^2}\simeq \int_{\p{\Omega^\N}^\Z}\d x \sqrt{\sum_{j\in \Z} \sum_{i\in\N} \left|f_{i,j}\p{x_i^{(j)},s_j}\right|^2}.\eeq
The variables $x^{(j)}_i$ are independent, so by Johnson-Schechtman inequality, there are $a_{i,j}:\Omega\times [0,1]^\infty\to\{0,1\}$ such that
\begin{eqnarray} \int_{\p{\Omega^\N}^\Z}\d x \sqrt{\sum_{j\in \Z} \sum_{i\in\N} \left|f_{i,j}\p{x_i^{(j)},s_j}\right|^2}&\gtrsim &\sum_{i,j} \int_{\Omega}\d \xi \left|a_{i,j}\p{\xi,s}f_{i,j}\p{\xi,s_j}\right| \\
&\ & +  \sqrt{\sum_{i,j} \int_{\Omega}\d \xi \left|\p{1-a_{i,j}}\p{\xi,s}f_{i,j}\p{\xi,s_j}\right|^2}. \end{eqnarray}
By a standard application of Theorem \ref{premainlemma} for the second summand and trivially for the first, 
\beq \int_{[0,1]^\infty}\d s \left( \sum_{i,j} \int_{\Omega}\d \xi \left|a_{i,j}\p{\xi,s}f_{i,j}\p{\xi,s_j}\right|+ \sqrt{\sum_{i,j} \int_{\Omega}\d \xi \left|\p{1-a_{i,j}}\p{\xi,s}f_{i,j}\p{\xi,s_j}\right|^2}\right)\gtrsim\eeq
\beq  \int_{[0,1]^\infty}\d s \left( \sum_{i,j} \int_{\Omega}\d \xi \left|\tilde{a}_{i,j}f_{i,j}\p{\xi,s_j}\right|+ \sqrt{\sum_{i,j} \int_{\Omega}\d \xi \left|\p{1-\tilde{a}_{i,j}}f_{i,j}\p{\xi,s_j}\right|^2}\right)\eeq
for $\tilde{a}_{i,j}= \mathbbm{1}_{\left\{\E_j a_{i,j}\geq \frac12\right\}}$, where $\E_j$ is taken with respect to $s$. Since
\beq \int_{[0,1]^\infty}\d s  \sum_{i,j} \int_{\Omega}\d \xi \left|\tilde{a}_{i,j}\p{\xi,s}f_{i,j}\p{\xi,s_j}\right|= \sum_{i,j} \int_{\Omega}\d \xi \int_{[0,1]}\d \sigma \left|\tilde{a}_{i,j}\p{\xi,\sigma}f_{i,j}\p{\xi,\sigma}\right|,\eeq
the first summand no longer features integration over an infinite product. Applying $L^2\p{\bigsqcup_i \Omega}$-valued Johnson-Schechtman inequality to the second summand gives $b_j: [0,1]\to \{0,1\}$ such that 
\begin{align} \int_{[0,1]^\infty}\d s \sqrt{\sum_{i,j} \int_{\Omega}\d \xi \left|\p{1-\tilde{a}_{i,j}}f_{i,j}\p{\xi,s_j}\right|^2}\gtrsim & \sum_j \int_{[0,1]}\d \sigma \sqrt{\sum_i\int_{\Omega}\d \xi \left|b_j\p{1-\tilde{a}_{i,j}}f_{i,j}\p{\xi,\sigma} \right|^2}\\
&+ \sqrt{\sum_j \int_{[0,1]}\d \sigma \sum_i \int_{\Omega}\d \xi\left|\p{1-b_j}\p{1-\tilde{a}_{i,j}}f_{i,j}\p{\xi,\sigma}\right|^2}.\end{align}
Ultimately, taking $\alpha_{i,j}= \tilde{a}_{i,j}f_{i,j}$, $\beta_{i,j}= b_{i,j}\p{1-\tilde{a}_{i,j}}f_{i,j}$, $\gamma_{i,j}=\p{1-b_{i,j}}\p{1-\tilde{a}_{i,j}}f_{i,j}$ we get
\begin{align}\label{eq:kfunv1v2} K\p{f,1;L^1\p{\Omega^\infty,L^1\p{\R}},L^2\p{\Omega^\infty,L^2\p{\R}}} \gtrsim & \sum_{i,j} \int_{\Omega^2}\d \xi \d \sigma \left|\alpha_{i,j}\p{\xi,\sigma}\right|\\ 
& + \sum_{j}\int_{[0,1]}\d \sigma \sqrt{\sum_i\int_{\Omega}\d \xi \left|\beta_{i,j}\p{\xi,\sigma}\right|^2}\\
&+ \sqrt{\sum_{j}\int_{[0,1]}\d \sigma \sum_i\int_{\Omega}\d \xi \left|\gamma_{i,j}\p{\xi,\sigma}\right|^2 }\end{align}
and
\beq \alpha_{i,j}+\beta_{i,j}+\gamma_{i,j}= f_{i,j}.\eeq
In order to get a decomposition for the $K$-functional of Hoeffding subspaces, we put $g=\bigsqcup_j g_j$, $h=\bigsqcup_{j} h_j$, where
\beq g_j(x,\sigma)=\sum_i \p{\alpha_{i,j}+\beta_{i,j}}\p{x_i,\sigma}, \quad h_j\p{x,\sigma}= \sum_i \gamma_{i,j}\p{x_i,\sigma}.\eeq
Obviously, 
\beq \left\|h\right\|_{\V^2_1\p{L^2}}= \sqrt{\sum_{i,j}\left\|\gamma_{i,j}\right\|_{L^2\p{\Omega\times[0,1]}}^2},\eeq
because $\gamma_{i,j}$ are orthogonal. Moreover, by the trivial part of Johnson-Schechtman inequality,
\begin{align} \left\|g\right\|_{\V^1_1\p{L^1}}\leq & \sum_j \int_{[0,1]}\d\sigma \int_{\Omega^\infty}\d x\left|\sum_i \alpha_{i,j}\p{x_i,\sigma}\right| + \sum_j \int_{[0,1]}\d\sigma \int_{\Omega^\infty}\d x\left|\sum_i \beta_{i,j}\p{x_i,\sigma}\right| \\
\lesssim & \sum_j \int_{[0,1]}\d\sigma \p{ \sum_i\int_{\Omega}\d \xi \left|\alpha_{i,j}\p{\xi,\sigma}\right| + \sqrt{\sum_i\int_{\Omega}\d \xi \left|\beta_{i,j}\p{\xi,\sigma}\right|^2} },\end{align}
which plugged into \eqref{eq:kfunv1v2} proves that $f=g+h$ satisfies the desired inequality.\end{proof}

\end{document}